\documentclass[11pt]{article}
\usepackage{amsfonts}
\usepackage{amsmath,amssymb,amsthm,amsfonts,enumerate}
\usepackage{graphicx}
\usepackage{float}
\usepackage[scriptsize,nooneline]{subfigure}
\usepackage{indentfirst}
\usepackage{cite}
\usepackage{color}
\usepackage{mathrsfs}
\usepackage{epstopdf}
\usepackage{setspace}
\usepackage[labelfont=footnotesize,font=footnotesize]{caption}
\usepackage[colorlinks, citecolor=red]{hyperref}

\newcommand{\R}{{\mathbb R}}

\newtheorem{theorem}{Theorem}[section]

\newtheorem{lemma}{Lemma}[section]

\newtheorem{definition}{Definition}[section]

\newtheorem{proposition}{Proposition}[section]
\newtheorem{remark}{Remark}[section]

\begin{document}
\title{\bf Variable Fofana's Spaces and their Pre-dual
\thanks{The research was supported by the National Natural Science Foundation of China(12061069).}
}
\author{Fan Yang\thanks{ Author E-mail address: yf18209050612@163.com.},\quad Jiang Zhou\thanks{ Corresponding author E-mail address: zhoujiang@xju.edu.cn.}\\ \\[.5cm]
\small  College of Mathematics and System Sciences, Xinjiang University, Urumqi 830046\\
\small People's Republic of China
}
\date{}
\maketitle

{\bf Abstract:}{
\hyphenpenalty=5000
\tolerance=1000
In this paper, we introduce the variable Fofana's spaces $(L^{p(\cdot)},L^q)^\alpha (\mathbb{R}^n)$ where $1< p(\cdot)<\infty$ and $1\leq q,\alpha\leq\infty$, then show some properties and establish the pre-dual of those spaces which are contributed to prove the necessary conditions of fractional integral commutators' boundedness. As applications, the characterization of fractional integral operators and commutators on variable Fofana's spaces are discussed, which are new result even for the classical Fofana's spaces.


\baselineskip 15pt

\section{\hspace{-0.6cm}{\bf }~~Introduction \label{s1}}

The variabl Lebesgue spaces $L^{p(\cdot)}$ appeared in literature for the first time
already in 1931 by W. Orlicz[1]. The major study of this spaces was initiated by O. Kov\'{a}\v{c}ik and J. R\'{a}kosn\'{\i}k[2] in 1991, where basic properties such as Banach space, reflexivity, separability, uniform convexity, H$\rm\ddot{o}$lder inequalities and embeddings of type $L^p(\cdot)\hookrightarrow L^q(\cdot)$ were obtained in higher dimension Euclidean spaces. In 2001, X. Fan and D. Zhao[3] further studied of the results in [2]. In 2011, L. Diening et al.[4] summarized the properties of the variable Lebesgue spaces in a comprehensive way. Also there are recent many interesting and important papers appeared in variable Lebesgue spaces (see, [5], [6], [7] [8], [9])

In 1926, amalgam spaces $(L^1, \ell^2)(\R)$ and $(L^2, \ell^\infty)(\R)$ were first introduced by N. Wiener[10]. In 1975, F. Holland[11] conducted a systematic study on general amalgam spaces $(L^p, l^q)(\R^n)$. In 2012, the Wiener's amalgam spaces with variable exponent $(L^{p(x)},L_{\omega}^{q})(\R^n)$ and $(L_{\omega}^{p(x)},L_{\upsilon}^{q})(\R^n)$ were defined by \'{I}. Aydin and A. T. G\"{u}rkanli[12]. \par
Given $1\leq p,q\leq\infty$, amalgam spaces $(L^p, l^q)(\R^n)$ is equipped with the norm $\|f\|_{p,q}=\left\|\left\{\|f\chi_{\mathrm{I}_k}\|_{L^p(\R^n)}\right\}\right\|_{\ell^q(\R^n)}$, where $\mathrm{I}_k=\prod^n_{j=1}[k_j,(k_j+1))$, $k=(k_j)_{1\leq j\leq\infty}\in \mathbb{Z}^n$.

In fact, for any $r>0$, the dilation operator $St_{r}^{(p)}:f\mapsto r^{-\frac{n}{p}}f(r^{-1}\cdot)$ is isometric on $L^{p}(\R^n)$. However, amalgam spaces do not have this property. If $p\neq q$, there does not exist $\alpha$ such that $\sup_{r>0}\|St_{r}^{(\alpha)}(f)\|_{p,q}<\infty$, although $St_{r}^{(\alpha)}(f)\in(L^{p},\ell^q)(\R^n)$ for all $f\in(L^p,\ell^q)(\R^n)$, $r>0$ and $\alpha>0$. In 1988, I. Fofana compensated above shortcomings, in [13] introduced the Fofana's spaces $(L^p,\ell^s)^\alpha(\R^n)$, which consist of $f\in(L^p,\ell^q)(\R^n)$ and satisfying $\sup_{r>0}\|St_{r}^{(\alpha)}(f)\|_{p,q}<\infty$,
In addition a continuous form of the Fofana's spaces $(L^{p},L^{q})^\alpha(\mathbb{R}^n)$ as follow:
$$(L^{p},L^{q})^\alpha(\mathbb{R}^n)=\{f\in L_{loc}^{p}(\mathbb{R}^n):\|f\|_{(L^{p},L^{q})^\alpha(\mathbb{R}^n)}<\infty\}, 1\leq p,q,\alpha\leq\infty,$$
where
$$\|f\|_{(L^{p},L^{q})^\alpha(\mathbb{R}^n)}=\sup_{r>0}\bigg\||B(\cdot,r)|^{\frac{1}{\alpha}-\frac{1}{p}-\frac{1}{q}}\|f\chi_{B(\cdot,r)}\|_{L^p(\mathbb{R}^n)}\bigg\|_{L^q(\mathbb{R}^n)},$$
$B(x,r)=\{y\in\mathbb{R}^n:|y-x|<r\}$ and use the notation $|B|$ for the Lebesgue measure of ball $B$.\par

In 2019, the pre-dual of the Fofana's spaces introduced by H. G. Feichtinger and J. Feuto[14].
In 2022, the mixed-norm amalgam spaces and their pre-dual were introduced by H. Zhang and J. Zhou[15], and then the characterization of fractional integral operators and commutators on mixed-norm amalgam spaces were also obtained.

The fractional integral operators $I_{\gamma}$ are defined by
$$I_\gamma f(x)=\int_{\mathbb{R}^n}\frac{f(y)}{|x-y|^{n-\gamma}}dy,~~0<\gamma<n.$$

For a locally integrable function $b$, the commutators of fractional integral operators $[b,I_\gamma]$ are defined by
$$[b,I_\gamma]f(x):=b(x)I_\gamma f(x)-I_\gamma(bf)(x)=\int_{\mathbb{R}^n}\frac{(b(x)-b(y))f(y)}{|x-y|^{n-\gamma}}dy,~~0<\gamma<n.$$\par

The space $\rm BMO(\R^n)$ consists of all locally integrable functions $f$ such that
$$\|f\|_{BMO(\R^n)}:=\sup_{B}\frac{1}{|B|}\int_B|f(x)-f_B|dx<\infty,$$
where the supremum is taken over all balls $B\in \mathbb{R}^n$, $f_B$ is the mean of $f$ on $B$, $|B|$ denotes the Lebesgue measure
of $B$.\par

In 1970, E. M. Stein[18] showed that the fractional integrals operator $I_\gamma$ is bounded on Lebesgue spaces. In 2007, C. Capone et al.[19] discussed the boundedness of fractional integrals on variable Lebesgue spaces. In 1982, Chanillo[20] had initially introduced the commutator $[b, I_\gamma]$ with $b \in \rm BMO$ and proved the boundedness on Lebesgue spaces. And in 2010, Izuki[21] generalizes this result to the case of variable Lebesgue spaces.

In this paper, we introduce the variable Fofana's spaces $(L^{p(\cdot)},L^q)^\alpha (\mathbb{R}^n)$($1< p(\cdot)<\infty$, $1\leq q,\alpha\leq\infty$) and show some basic properties. Naturally, it will be a very interesting problem to ask whether we can establish the pre-dual of variable Fofana's spaces and then obtain the characterization of fractional integral operators and commutators on variable Fofana's spaces where $b\in \rm BMO(\mathbb{R}^n)$.\par

Firstly, we recall some standard notations and lemmas in variable Lebesgue spaces. Given an open set $E\subset\mathbb{R}^n$ and consider a measurable function $p(\cdot):E \rightarrow[1,\infty)$, $L^{p(\cdot)}(E)$ denotes the set of measurable functions $f$ on $E$ such that for
some $\lambda>0$,
$$\int_{E}\bigg(\frac{|f(x)|}{\lambda}\bigg)^{p(x)}dx<\infty,$$
this set becomes a Banach function space when equipped with the Luxemburg-Nakano norm\\
$$\|f\|_{L^{p(\cdot)}(E)}=\inf\bigg\{\lambda>0:\int_{E}\bigg(\frac{|f(x)|}{\lambda}\bigg)^{p(x)}dx\leq1\bigg\}.$$
These spaces are variable $L^p$ spaces, since they generalize the standard $L^p$
spaces: if $p(\cdot)=p$ is constant, then $L^{p(\cdot)}(E)$ is isometrically isomorphic to $L^p(E)$. (Here and below we write $p(\cdot)$ instead of $p$ to emphasize that the exponent is function and a constant.)\par
The spaces $L_{loc}^{p(\cdot)}(E)$ are defined by
\begin{center}
$L_{loc}^{p(\cdot)}(E):=\{f:f\in L^{p(\cdot)}(F)$ for all compact subsets $F\subset E\},$
\end{center}
 and define $\mathcal{P}(E)$ to be the set of $p(\cdot): E\rightarrow[1, \infty)$ such that
\begin{center}
$p_-={\rm ess}\inf\{p(x):x\in E\}>1,~~~~p^+={\rm ess}\sup\{p(x):x\in E\}<\infty.$
\end{center}\par
Denote $p'(x)=p(x)/(p(x)-1)$ and $\mathcal{B}(E)$ be the set of $p(\cdot)\in\mathcal{P}(E)$ such that the Hardy--Littlewood maximal operators $M$ are bounded on $L^{p(\cdot)}(E)$.\par

Throughout this paper, given a ball $B:=B(x,r)=\{y\in \R^n: |y-x|<r\}$ and we use the notation $|B|$ for the Lebesgue measure of ball $B$ and $\chi_B$ means the characteristic function for a measurable set $B\subset\R^n$, and $aB=B(x,ar)$, $a > 0$. A symbol C always means a positive constant independent of the main parameters and may change from one occurrence to another. $C'\sim C''$ means that $C'$ is equivalent to $C''$, that is, $C'\lesssim C''(C'\le CC'')$ and $C''\lesssim C'(C''\le CC')$.

In variable Lebesgue spaces, there are some important lemmas as follow.

\begin{lemma}[see\cite{6} Theorem 1.1]\label{l1.1}
Given an open set $E\subset\mathbb{R}^n$ and $p(\cdot)\in\mathcal{P}(E)$, suppose that $p(\cdot)$ satisfies
\begin{align}\label{1}
|p(x)-p(y)|\leq\frac{C}{-\log(|x-y|)},~~~~x,y\in E, |x-y|\leq1/2
\end{align}
and
\begin{align}\label{2}
|p(x)-p(y)|\leq\frac{C}{\log(|x|+e)},~~~~x,y\in E, |y|\geq|x|,
\end{align}
then $p(\cdot)\in\mathcal{B}(E)$, that is, the Hardy--Littlewood maximal operator $M$ is bounded on $L^{p(\cdot)}(E)$.
\end{lemma}

\begin{lemma}[see\cite{2} Theorem 2.1]\label{l1.2}
Let $p(\cdot)\in\mathcal{P}(E)$. If $f\in L^{p(\cdot)}(E)$ and $g\in L^{p'(\cdot)}(E)$, then fg is integrable on $E\subset\R^n$ and
$$\int_{E}|f(x)g(x)|dx\leq r_p\|f\|_{L^{p(\cdot)}(E)}\|g\|_{L^{p'(\cdot)}(E)},$$
where
$$r_p=1+1/p_--1/p^+.$$
\end{lemma}\par
This inequality is named the generalized H\"{o}lder inequality with respect to the variable Lebesgue spaces.

\begin{lemma}[see\cite{4} Corollary 4.5.9]\label{l1.3}
Let $p(\cdot)\in\mathcal{P}(\mathbb{R}^n)$ satisfy conditions (1) and (2) in Lemma1.1. Then $\|\chi_Q\|_{L^{p(\cdot)}(\mathbb{R}^n)}\sim|Q|^{\frac{1}{p_Q}}$ for every cube(or ball) $Q\subset\mathbb{R}^n$. More concretely,\\
\begin{eqnarray*}
\|\chi_Q\|_{L^{p(\cdot)}(\mathbb{R}^n)}\sim
\begin{cases}
|Q|^{\frac{1}{p(x)}}      ~~~~~~if~|Q|\leq 2^n~and~x\in Q,\\
|Q|^{\frac{1}{p(\infty)}} ~~~~~if~|Q|\geq 1
\end{cases}
\end{eqnarray*}
for every cube (or ball) $Q\subset\mathbb{R}^n$, where $p(\infty)=\lim\limits_{x\rightarrow\infty}p(x)$ and $p_Q$ is the mean of $p$ on $Q$.
\end{lemma}\par

\begin{lemma}[see\cite{16} Lemma 2.9]\label{l1.4}
Suppose $p(\cdot)\in\mathcal{B}(\mathbb{R}^n)$. Then there exists constant $C>0$ such that for all balls B in $\mathbb{R}^n$,
$$\frac{1}{|B|}\|\chi_{B}\|_{L^{p(\cdot)}(\mathbb{R}^n)}\|\chi_{B}\|_{L^{p'(\cdot)}(\mathbb{R}^n)}\leq C.$$
\end{lemma}\par


\section{\hspace{-0.6cm}{\bf }~~Variable Fofana's spaces $(L^{p(\cdot)},L^q)^\alpha(\mathbb{R}^n)$\label{s2}}

In this section, we introduce the definition and the properties of the $(L^{p(\cdot)},L^q)^\alpha(\mathbb{R}^n)$.\par
The variable amalgam spaces $(L^{p(\cdot)},L^{q})(\mathbb{R}^n)$ as follow:
$$(L^{p(\cdot)},L^{q})(\mathbb{R}^n):=\left\{f\in L_{loc}^{p(\cdot)}(\mathbb{R}^n):\|f\|_{(L^{p(\cdot)},L^{q})(\mathbb{R}^n)}:=\left\|\|f\chi_{B(\cdot,1)}\|_{L^{p(\cdot)}(\R^n)} \right\|_{L^{q}(\R^n)}<\infty\right\}.$$

\begin{definition}
Let $p(\cdot)\in \mathcal{P}(\R^n) $, $1\leq q, \alpha\leq\infty$. If $f\in L_{loc}^{p(\cdot)}(\mathbb{R}^n)$, then
\begin{align*}
(L^{p(\cdot)},L^{q})^\alpha(\mathbb{R}^n)=\{f:\|f\|_{(L^{p(\cdot)},L^{q})^\alpha(\mathbb{R}^n)}<\infty\},
\end{align*}
where
\begin{align*}
\|f\|_{(L^{p(\cdot)},L^q)^\alpha(\mathbb{R}^n)}
&=\sup_{r>0}\left\{\int_{\mathbb{R}^n}\left(|B(x,r)|^{\frac{1}{\alpha}-\frac{1}{p(x)}-\frac{1}{q}}\|f\chi_{B(\cdot,r)}\|_{L^{p(\cdot)}(\mathbb{R}^n)}\right)^qdx\right\}^{\frac{1}{q}}\\
&=\sup_{r>0}\left\||B(\cdot,r)|^{\frac{1}{\alpha}-\frac{1}{p(\cdot)}-\frac{1}{q}}\|f\chi_{B(\cdot,r)}\|_{L^{p(\cdot)}(\mathbb{R}^n)}\right\|_{L^q(\mathbb{R}^n)}.
\end{align*}
\end{definition}\par
Next, we claim some propositions.
\begin{proposition}\label{p2.1}
We can get some interesting facts about $(L^{p(\cdot)}, L^q)^\alpha(\mathbb{R}^n)$\rm:\par
{\rm (i)} For $p(\cdot):\R^n \rightarrow[1,\infty)$, $1\leq q, \alpha\leq\infty$. If $ p(\cdot)\leq\alpha\leq q$, then $(L^{p(\cdot)}, L^q)^\alpha(\mathbb{R}^n)\hookrightarrow(L^{p(\cdot)}, L^q)(\mathbb{R}^n)$.\par
{\rm(ii)} If $p(x)=\alpha$ and $q=\infty$, then $(L^{p(\cdot)}, L^q)^\alpha(\mathbb{R}^n)$ is just the
        usual variable $L^{\alpha}(\mathbb{R}^n)$ space .
\end{proposition}

\begin{proof}
By direct calculation, we have\par
(i) When $p(\cdot)\leq\alpha\leq q$,
\begin{align*}
\|f\|_{(L^{p(\cdot)},L^q)}
&\sim\left\||B(\cdot,1)|^{\frac{1}{\alpha}-\frac{1}{p(\cdot)}-\frac{1}{q}} \|f\chi_{B(\cdot,1)}\|_{L^{p(\cdot)}(\R^n)} \right\|_{L^{q}(\R^n)}\\
&\leq\sup_{r>0 }\left\||B(\cdot,r)|^{\frac{1}{\alpha}-\frac{1}{p(\cdot)}-\frac{1}{q}} \|f\chi_{B(\cdot,r)}\|_{L^{p(\cdot)}(\R^n)}\right\|_{L^{q}(\R^n)}\\
&=\|f\|_{(L^{p(\cdot)},L^{q})^{\alpha}(\R^n)}<\infty.
\end{align*}
Therefore, $(L^{p(\cdot)},L^{q})^\alpha(\R^n)\hookrightarrow(L^{p(\cdot)},L^{q})(\R^n)$ with $\|f\|_{(L^{p(\cdot)},L^{q})}\le\|f\|_{(L^{p(\cdot)},L^{q})^{\alpha}}$.\par
(ii) When $p(\cdot)=\alpha$ and $q=\infty$, it is apparent.
\end{proof}

\begin{proposition}\label{p2.2}
Let $p(\cdot)\in \mathcal{P}(\R^n), 1\leq q, \alpha\leq\infty$. Then $(L^{p(\cdot)}, L^q)^\alpha(\mathbb{R}^n)$ are Banach spaces.
\end{proposition}

\begin{proof}
 First, we will check the triangle inequality. For $f, g\in(L^{p(\cdot)}, L^q)^\alpha(\mathbb{R}^n)$,
\begin{align*}
\|f+g\|_{(L^{p(\cdot)}, L^q)^\alpha(\mathbb{R}^n)}
&=\sup_{r>0}\left\||B|^{\frac{1}{\alpha}-\frac{1}{p(\cdot)}-\frac{1}{q}}\|(f+g)\chi_{B}\|_{L^{p(\cdot)}(\mathbb{R}^n)}\right\|_{L^q(\mathbb{R}^n)}\\
&\leq\sup_{r>0}\left\||B|^{\frac{1}{\alpha}-\frac{1}{p(\cdot)}-\frac{1}{q}}\|f\chi_{B}\|_{L^{p(\cdot)}(\mathbb{R}^n)}\right\|_{L^q(\mathbb{R}^n)}
     +\sup_{r>0}\left\||B|^{\frac{1}{\alpha}-\frac{1}{p(\cdot)}-\frac{1}{q}}\|g\chi_{B}\|_{L^{p(\cdot)}(\mathbb{R}^n)}\right\|_{L^q(\mathbb{R}^n)}\\
&=\|f\|_{(L^{p(\cdot)}, L^q)^\alpha(\mathbb{R}^n)}+\|g\|_{(L^{p(\cdot)}, L^q)^\alpha(\mathbb{R}^n)}.
\end{align*}\par
The positivity and the homogeneity are both clear. Thus, we prove that $(L^{p(\cdot)}, L^q)^\alpha(\mathbb{R}^n)$ are spaces with norm $\|\cdot\|_{(L^{p(\cdot)}, L^q)^\alpha(\mathbb{R}^n)}$.\par
It remains to check the completeness. Without loss the generality, let a Cauchy sequence
$\{f_j\}_{j=1}^{\infty}\subset (L^{p(\cdot)}, L^q)^\alpha(\mathbb{R}^n)$ satisfy $$\|f_{j+1}-f_j\|_{(L^{p(\cdot)}, L^q)^\alpha(\mathbb{R}^n)}<2^{-j}.$$\par
We write $g=|f_1|+\sum_{j=1}^{\infty}|f_{j+1}-f_j|$. Then
$$\|g\|_{(L^{p(\cdot)}, L^q)^\alpha(\mathbb{R}^n)}\le \|f_1\|_{(L^{p(\cdot)}, L^q)^\alpha(\mathbb{R}^n)}+\sum_{j=1}^{\infty}\|f_{j+1}-f_{j}\|_{(L^{p(\cdot)}, L^q)^\alpha(\mathbb{R}^n)}<\infty,$$
for almost everywhere $x\in\mathbb{R}^n$,
$|g(x)|=|f_1(x)|+\sum_{j=1}^{\infty}|f_{j+1}(x)-f_j(x)|<\infty.$\par
Thus if $f=f_1+\sum_{j=1}^{\infty}(f_{j+1}-f_j)=\lim\limits_{j\rightarrow\infty}f_j$, then $|f|\leq |g|\in(L^{p(\cdot)}, L^q)^\alpha(\mathbb{R}^n)$.
We write $f_J=f_1+\sum_{j=1}^{J-1}(f_{j+1}-f_j)$, furthermore,
\begin{align*}
\|f-f_{J}\|_{(L^{p(\cdot)}, L^q)^\alpha(\mathbb{R}^n)}
&=\|\sum^{\infty}_{j=1}(f_{j+1}-f_{j})-\sum^{J-1}_{j=1}(f_{j+1}-f_{j})\|_{(L^{p(\cdot)}, L^q)^\alpha(\mathbb{R}^n)}\\
&\leq \sum^{\infty}_{j=J}\|f_{j+1}-f_{j}\|_{(L^{p(\cdot)}, L^q)^\alpha(\mathbb{R}^n)}
\leq 2\cdot 2^{-J}
\end{align*}
and
$$\lim\limits_{J\rightarrow\infty}\|f-f_{J}\|_{(L^{p(\cdot)}, L^q)^\alpha(\mathbb{R}^n)}=0.$$\par
So we prove that $(L^{p(\cdot)}, L^q)^\alpha(\mathbb{R}^n)$ are Banach spaces.
\end{proof}

\begin{lemma}\label{L2.1}
Let $f\in L_{loc}^1(\mathbb{R}^n)$ and $p(\cdot)\in \mathcal{P}(\R^n)$, then
$$\lim_{r\rightarrow0}\frac{\|f\chi_{B}\|_{L^{p(\cdot)(\mathbb{R}^n)}}}{\|\chi_{B}\|_{L^{p(\cdot)}(\mathbb{R}^n)}}
   =|f(x)|~~~~a.e.~~x\in\mathbb{R}^n.$$
\end{lemma}

\begin{proof}
Assume $1<p(\cdot)\leq p^+<\infty$. By Lemma \ref{l1.3}, easy to know that

$$\|\chi_{B}\|_{L^{p(\cdot)}(\mathbb{R}^n)}\sim\|\chi_{B}\|_{L^{p'(\cdot)}(\mathbb{R}^n)}^{-1}|B|.$$\par

Using the generalized H\"{o}lder inequality,\\
\begin{align*}
\frac{1}{|B|}\int_{B}|f(y)|dy
\lesssim \frac{1}{|B|}\|f\chi_{B}\|_{L^{p(\cdot)(\mathbb{R}^n)}}\|\chi_{B}\|_{L^{p'(\cdot)}(\mathbb{R}^n)}
\lesssim \frac{\|f\chi_{B}\|_{L^{p(\cdot)(\mathbb{R}^n)}}}{\|\chi_{B}\|_{L^{p(\cdot)}(\mathbb{R}^n)}}.
\end{align*}\par
By Lemma1.3, we also have
\begin{align*}
\frac{\|f\chi_{B}\|_{L^{p(\cdot)}(\mathbb{R}^n)}}{\|\chi_{B}\|_{L^{p(\cdot)}(\mathbb{R}^n)}}
\sim\frac{1}{|B|^{\frac{1}{p(\cdot)}}}\|f\chi_{B}\|_{L^{p(\cdot)}(\mathbb{R}^n)}
\lesssim \frac{1}{|B|^{\frac{1}{p^+}}}\|f\chi_{B}\|_{L^{p^+}(\mathbb{R}^n)}
=\bigg(\frac{1}{|B|}\int_{B}|f(y)|^{p^+}dy\bigg)^{\frac{1}{p^+}}.
\end{align*}\par
Thus,
$$\frac{1}{|B|}\int_{B}|f(y)|dy
\lesssim \frac{\|f\chi_{B}\|_{L^{p(\cdot)(\mathbb{R}^n)}}}{\|\chi_{B}\|_{L^{p(\cdot)}(\mathbb{R}^n)}}
\lesssim \bigg(\frac{1}{|B|}\int_{B}|f(y)|^{p^+}dy\bigg)^{\frac{1}{p^+}}.$$

By Lebesgue differential theorem,
\begin{align*}
\lim_{r\rightarrow0}\|\chi_{B}\|_{L^{p(\cdot)}(\mathbb{R}^n)}^{-1}\|f\chi_{B}\|_{L^{p(\cdot)}(\mathbb{R}^n)}
   =|f(x)|~~~~a.e.~~x\in\mathbb{R}^n.
\end{align*}\par
The proof is complete.
\end{proof}
\begin{proposition}\label{p2.3}
Let $p(\cdot)\in \mathcal{P}(\R^n)$, $(L^{p(\cdot)}, L^q)^\alpha(\mathbb{R}^n)$ are nontrivial if and only if $p(\cdot)\leq \alpha \leq q$.
\end{proposition}

\begin{proof}
Suppose the spaces $(L^{p(\cdot)}, L^q)^\alpha(\mathbb{R}^n)$ are nontrivial. We prove these by contradiction. In fact, by Lemma \ref{L2.1}, we know
$$\lim_{r\rightarrow0}\frac{\|f\chi_{B}\|_{L^{p(\cdot)(\mathbb{R}^n)}}}{\|\chi_{B}\|_{L^{p(\cdot)}(\mathbb{R}^n)}}
   =|f(x)|~~~~a.e.~~x\in\mathbb{R}^n.$$\par
Thus, if $\alpha > q$ and $f\neq 0$, then
$$ \lim_{r\rightarrow 0}|B|^{\frac{1}{\alpha}-\frac{1}{q}}
\frac{\|f\chi_{B}\|_{L^{p(\cdot)(\mathbb{R}^n)}}}{\|\chi_{B}\|_{L^{p(\cdot)}(\mathbb{R}^n)}}=\infty.$$\par
Therefor, we prove $\alpha\leq q$.

If $\alpha<p(\cdot)$, then $\alpha<p_- $ and $\|\chi_{B(x_{0},r_{0})}\|_{(L^{p(\cdot)},L^{q})^{\alpha}(\mathbb{R}^n)}=\infty$ for any ball $B(x_0,r_0)$, which show that $(L^{p(\cdot)},L^{q})^{\alpha}(\mathbb{R}^n)$ are trivial spaces. Hence, we acquire $\alpha\geq p(\cdot)$. Indeed, if $x\in B(x_0,\frac{r}{2})$ and $2r_0<r$, then for any $y\in B(x_0,r_0)$, we have
$$|x-y|\leq|x_0-x|+|x_0-y|\leq\frac{r}{2}+r_0<r,$$
that is $B(x_0,r_0)\subset B(x,r)$. Therefore,
\begin{align*}
\|\chi_{B(x_{0},r_{0})}\|_{(L^{p(\cdot)},L^{q})^{\alpha}(\mathbb{R}^n)}
&\sim\sup_{r>0}
\left\| r^{\frac{n}{\alpha}-\frac{n}{p(\cdot)}-\frac{n}{q}}\|\chi_{B(x_{0},r_0)}\chi_{B(\cdot,r)}\|_{L^{p(\cdot)}(\mathbb{R}^n)} \right\|_{L^q(\mathbb{R}^n)}\\
&\geq\sup_{r>2r_0}r^{\frac{n}{\alpha}-\frac{n}{p_-}-\frac{n}{q}}
\left\|\chi_{B(x_0,\frac{r}{2})}\|\chi_{B(x_{0},r_0)}\|_{L^{p(\cdot)}(\mathbb{R}^n)} \right\|_{L^{q}(\mathbb{R}^n)}\\
&\gtrsim\sup_{r>2r_0} r^{\frac{n}{\alpha}-\frac{n}{p_-}-\frac{n}{q}}\cdot r^{\frac{n}{q}}\\
&\geq\lim_{r\rightarrow+\infty} r^{\frac{n}{\alpha}-\frac{n}{p_-}}
=+\infty.
\end{align*}

For the opposite side, it is easy to prove $\chi_{B(0,1)}\in(L^{p(\cdot)},L^{q})^{\alpha}(\mathbb{R}^n)$ if $p(\cdot)\leq \alpha \leq q$. Obviously,
$$\|\chi_{B(0,1)}\|_{(L^{p(\cdot)},L^{q})^{\alpha}(\mathbb{R}^n)} \sim\sup_{r>0}
\left\|r^{\frac{n}{\alpha}-\frac{n}{p(\cdot)}-\frac{n}{q}} \|\chi_{B(0,1)}\chi_{B(\cdot,r)}\|_{L^{p(\cdot)}(\mathbb{R}^n)} \right\|_{L^{q}(\mathbb{R}^n)}.$$\par
If $r>1$, by $\frac{1}{\alpha}-\frac{1}{p(\cdot)}\leq\frac{1}{\alpha}-\frac{1}{p_+}\le 0$, we obtain
\begin{align*}
\quad\sup_{r>1}
\left\|r^{\frac{n}{\alpha}-\frac{n}{p(\cdot)}-\frac{n}{q}} \|\chi_{B(0,1)}\chi_{B(\cdot,r)}\|_{L^{p(\cdot)}(\mathbb{R}^n)} \right\|_{L^{q}(\mathbb{R}^n)}
&\le\sup_{r>1}r^{\frac{n}{\alpha}-\frac{n}{p_+}-\frac{n}{q}}
\left\| \|\chi_{B(0,1)}\|_{L^{p(\cdot)}(\mathbb{R}^n)}\cdot\chi_{B(0,r+1)}\right\|_{L^{q}(\mathbb{R}^n)}\\
&\lesssim \sup_{r>1}r^{\frac{n}{\alpha}-\frac{n}{p_+}-\frac{n}{q}}
(r+1)^{\frac{n}{q}}\\
&\lesssim \sup_{r>1}r^{\frac{n}{\alpha}-\frac{n}{p_+}}\\
&<\infty.
\end{align*}\par
For $r\leq 1$, then by $\frac{1}{\alpha}-\frac{1}{q}\ge 0$ and Lemma \ref{l1.3}, we have
\begin{align*}
\quad\sup_{r\leq 1}
\left\|r^{\frac{n}{\alpha}-\frac{n}{p(\cdot)}-\frac{n}{q}} \|\chi_{B(0,1)}\chi_{B(\cdot,r)}\|_{L^{p(\cdot)}(\mathbb{R}^n)} \right\|_{L^{q}(\mathbb{R}^n)}
&\lesssim\sup_{r> 1}
\left\|r^{\frac{n}{\alpha}-\frac{n}{q}}\cdot |B|^{-\frac{1}{p(x)}}\cdot \|\chi_{B(\cdot,r)}\|_{L^{p(\cdot)}(\mathbb{R}^n)}\cdot\chi_{B(0,r+1)}\right\|_{L^{q}(\mathbb{R}^n)}\\
&\lesssim\sup_{r> 1} r^{\frac{n}{\alpha}-\frac{n}{q}}\|\chi_{B(0,r+1)}\|_{L^{q}(\mathbb{R}^n)}\\
&\lesssim \sup_{r>0}r^{\frac{n}{\alpha}-\frac{n}{q}}\cdot(r+1)^{\frac{n}{q}}\\
&<\infty.
\end{align*}

\end{proof}

\section{\hspace{-0.6cm}{\bf }~~The pre-dual of variable Fofana's spaces\label{s3}}

In this section, we give the equivalent norms of variable Fofana's spaces. Let $Q_{r,k}=r[k+[0,1)^n]$ ($k\in\mathbb{Z}^n$) and
$\left\|\{a_k\}_{k\in \mathbb{Z}^n}\right\|_{\ell^{q}} :=\left(\sum_{k\in\mathbb{Z}^n}|a_k|^{q}\right)^{\frac{1}{q}}.$

\begin{proposition}\label{p3.1}
Let $p(\cdot)\in \mathcal{P}(\R^n), 1\leq q, \alpha\leq\infty$ and $p(\cdot)\leq \alpha \leq q$. We define ``discrete" variable Fofana's spaces.

\begin{align*}
(L^{p(\cdot)},\ell^{q})(\mathbb{R}^n)&:=\left\{f\in L_{loc}^{p(\cdot)}(\mathbb{R}^n): \|f\|_{p(\cdot),q,} :=\big\|\{\|f\chi_{Q_{1,k}}\|_{L^{p(\cdot)}}\}_{k\in\mathbb{Z}^n}\big\|_{\ell^{q}}\right\},\\
(L^{p(\cdot)},\ell^{q})^\alpha(\R^n)&:=\left\{f\in L_{loc}^{p(\cdot)}(\R^n):\|f\|_{p(\cdot),q,\alpha} :=\sup_{r>0}r^{\frac{n}{\alpha}-N_{r,p}}{_r\|f\|_{p(\cdot),q}}<\infty\right\},
\end{align*}
where
\begin{eqnarray*}
N_{r,p}=
\begin{cases}
\frac{n}{p_-}  ~~~~~~r>1,\\
\frac{n}{p^+}  ~~~~~r\leq1
\end{cases}
\end{eqnarray*}
and
$$_r\|f\|_{p(\cdot),q}:=\big\|\{\|f\chi_{Q_{r,k}}\|_{L^{p(\cdot)}}\}_{k\in\mathbb{Z}^n}\big\|_{\ell^{q}}.$$
Thus, we have
$$\|f\|_{(L^{p(\cdot)},L^{q})(\mathbb{R}^n)}\sim \|f\|_{{p(\cdot)},{q},},~\|f\|_{(L^{p(\cdot)},L^{q})^\alpha(\mathbb{R}^n)}\sim \|f\|_{{p(\cdot)},{q},\alpha}.$$
\end{proposition}

Before the proof of Proposition \ref{p3.1}, the following two lemmas are necessary.

\begin{lemma}\label{l3.1}
Let $p(\cdot)\in \mathcal{P}(\R^n), 1\leq q, \alpha\leq\infty$ and $p(\cdot)\leq \alpha \leq q$. For any constant $\rho\in (0,\infty)$, we have
$$\left\|\|f\chi_{B(\cdot,r)}\|_{L^{p(\cdot)}}\right\|_{L^{q}}\sim \left\|\|f\chi_{B(\cdot,\rho r)}\|_{L^{p(\cdot)}}\right\|_{L^{q}},$$
where the positive equivalence constant are independent of $f$.
\end{lemma}
\begin{proof}Firstly, we prove the lemma holds when $\rho>1$. It is obvious that
$$\left\|\|f\chi_{B(\cdot,r)}\|_{L^{p(\cdot)}}\right\|_{L^{q}}\le\left\|\|f\chi_{B(\cdot,\rho r)}\|_{L^{p(\cdot)}}\right\|_{L^{q}}.$$\par
Next, we prove the reverse inequality. It is easy to find $N\in\mathbb{N}$ and $\{x_1,x_2,\cdots,x_N\}$, such that
$$B(0,\rho r)\subset\bigcup_{j=1}^{N}B(x_j,r),$$
where $N$ is independent of $r$ and $N \sim 1$. Therefore, we have
$$\|f\chi_{B(x,\rho r)}\|_{L^{p(\cdot)}}\le\left\|f\sum_{j=1}^{N}\chi_{B(x+x_j,r)}\right\|_{L^{p(\cdot)}} \le\sum_{j=1}^{N}\left\|f\chi_{B(x+x_j,r)}\right\|_{L^{p(\cdot)}},$$
for any $x\in\R^n$. According to the translation invariance of the Lebesgue measure and $N \sim 1$, it follows that
$$\left\|\|f\chi_{B(\cdot,\rho r)}\|_{L^{p(\cdot)}}\right\|_{L^{q}} \le\sum_{j=1}^{N}\left\|\|f\chi_{B(\cdot+x_j,r)}\|_{L^{p(\cdot)}}\right\|_{L^{q}} \lesssim\left\|\|f\chi_{B(\cdot,r)}\|_{L^{p(\cdot)}}\right\|_{L^{q}}.$$\par

For the $\rho\in(0,1)$, we only need replace $r$ by $r/\rho$. The proof is completed.
\end{proof}

The following result plays an indispensable role in the proof of Proposition \ref{p3.1}.
\begin{lemma}\label{l3.2}
Let $p(\cdot)\in \mathcal{P}(\R^n), 1\leq q, \alpha\leq\infty$ and $p(\cdot)\leq \alpha \leq q$. Then we have
$$r^{\frac{n}{\alpha}-N_{p,r}}\left\|\left\{\|f\chi_{Q_{r,k}}\|_{L^{p(\cdot)}}\right\}_{k\in\mathbb{Z}^n}\right\|_{\ell^{q}}\sim \left\|r^{\frac{n}{\alpha}-\frac{n}{p(\cdot)}-\frac{n}{q}}\|f\chi_{B(\cdot,r)}\|_{L^{p(\cdot)}}\right\|_{L^{q}},$$
where the positive equivalence constants are independent of $f$ .
\end{lemma}
\begin{proof}
By the Lemma \ref{l3.1}, we only need show that\par
$$\left\|r^{\frac{n}{\alpha}-N_{p,r}}\left\{\|f\chi_{Q_{r,k}}\|_{L^{p(\cdot)}}\right\}_{k\in\mathbb{Z}^n}\right\|_{\ell^{q}}\sim \left\|r^{\frac{n}{\alpha}-\frac{n}{p(\cdot)}-\frac{n}{q}}\|f\chi_{B(\cdot,2\sqrt{n}r)}\|_{L^{p(\cdot)}}\right\|_{L^{q}}.$$\par
For any given $x\in\R^n$, we let $A_x:=\{k\in\mathbb{Z}^n:Q_{r,k}\cap B(x,2\sqrt{n}r)\neq\emptyset\}.$\par
Then the cardinality of $A_x$ is finite and $x\in B(rk,4\sqrt{n}r)$ for any $k\in A_x$. Thus,
\begin{align*}
\|f\chi_{B(x,2\sqrt{n}r)}\|_{L^{p(\cdot)}}&\le\left\|\sum_{k\in A_x}f\chi_{Q_{r,k}}\right\|_{L^{p(\cdot)}}
\le\sum_{k\in A_x}\left\|f\chi_{Q_{r,k}}\right\|_{L^{p(\cdot)}}
\le\sum_{k\in \mathbb{Z}^n}\left\|f\chi_{Q_{r,k}}\right\|_{L^{p(\cdot)}}\chi_{B(r k,4\sqrt{n}r)}(x).
\end{align*}\par
Taking $L^{q}$-norm on $x$, we have
$$\left\|r^{\frac{n}{\alpha}-\frac{n}{p(\cdot)}-\frac{n}{q}}\|f\chi_{B(\cdot,2\sqrt{n}r)}\|_{L^{p(\cdot)}}\right\|_{L^{q}}
\le\left\|r^{\frac{n}{\alpha}-\frac{n}{p(\cdot)}-\frac{n}{q}}\sum_{k\in \mathbb{Z}^n}\left\|f\chi_{Q_{r,k}}\right\|_{L^{p(\cdot)}}\chi_{B(r k,4\sqrt{n}r)}\right\|_{L^{q}}.$$\par
By the similar argument of Lemma \ref{l3.1}, there exist $N\in\mathbb{N}$ and $\{k_1,k_2,\cdots,k_N\}$, such that
$$B(0,4\sqrt{n}r)\subset\bigcup_{j=1}^{N}Q_{r,k_j},$$
where $N$ is independent of $r$ and $N\sim 1$.
According to the translation invariance of the Lebesgue measure, it follows that
\begin{align*}
\left\|r^{\frac{n}{\alpha}-\frac{n}{p(\cdot)}-\frac{n}{q}}\|f\chi_{B(\cdot,2\sqrt{n}r)}\|_{L^{p(\cdot)}}\right\|_{L^{q}}
&\le\left\|r^{\frac{n}{\alpha}-\frac{n}{p(\cdot)}-\frac{n}{q}}\sum_{j=1}^N\sum_{k\in \mathbb{Z}^n}\left\|f\chi_{Q_{r,k}}\right\|_{L^{p(\cdot)}}\chi_{Q_{r,k_j+k}}\right\|_{L^{q}}\\
&\lesssim r^{\frac{n}{\alpha}-N_{p,r}} \left\|\left\{\left\|f\chi_{Q_{r,k}}\right\|_{L^{p(\cdot)}}\right\}_{k\in\mathbb{Z}^n}\right\|_{\ell^{q}}.
\end{align*}\par
Indeed, the last inequality is obtained by the following fact that
\begin{align*}
\left(\int_{\R}\left| r^{\frac{n}{\alpha}-\frac{n}{p(\cdot)}-\frac{n}{q}}\sum_{k\in\mathbb{Z}^n}C_k\chi_{rk+(0,r]^n}(x)\right|^{q}dx\right)^{\frac{1}{q}}
&\lesssim r^{\frac{n}{\alpha}-N_{p,r}-\frac{n}{q}} \left(\int_{\R}\left|\sum_{k\in\mathbb{Z}^n}C_k\chi_{I_{k}}(x)\right|^{q}dx\right)^{\frac{1}{q}}\\
&\leq r^{\frac{n}{\alpha}-N_{p,r}-\frac{n}{q}}\left(\sum_{k\in\mathbb{Z}^n}\int_{I_{k}}\left|C_k\right|^{q}dx\right)^{\frac{1}{q}}\\
&\leq r^{\frac{n}{\alpha}-N_{p,r}}\left(\sum_{k\in\mathbb{Z}^n}\left|C_k\right|^{q}dx\right)^{\frac{1}{q}},
\end{align*}
where $C_k=\left\|f\chi_{Q_{r,k}}\right\|_{L^{p(\cdot)}}$ and $I_{k}=rk+[0,r)^n$$(k\in \mathbb{Z}^n)$. Thus, we prove that
$$\left\|r^{\frac{n}{\alpha}-\frac{n}{p(\cdot)}-\frac{n}{q}}\|f\chi_{B(\cdot,2\sqrt{n}r)}\|_{L^{p(\cdot)}}\right\|_{L^{q}}
\lesssim r^{\frac{n}{\alpha}-N_{p,r}} \left\|\left\{\left\|f\chi_{Q_{r,k}}\right\|_{L^{p(\cdot)}}\right\}_{k\in\mathbb{Z}^n}\right\|_{\ell^{q}}.$$\par
For the opposite inequality, it is obvious that
$$r^{\frac{n}{q}}\left\|\left\{\left\|f\chi_{Q_{r,k}}\right\|_{L^{p(\cdot)}}\right\}_{k\in\mathbb{Z}^n}\right\|_{\ell^{q}}
=\left\|\sum_{k\in\mathbb{Z}^n}\left\|f\chi_{Q_{r,k}}\right\|_{L^{p(\cdot)}}\chi_{Q_{r,k}}\right\|_{L^{q}}.$$
By $Q_{r,k}\subset B(x,2\sqrt{n}r)$ for $x\in Q_{r,k}$, we have
\begin{align*}
r^{\frac{n}{\alpha}-N_{p,r}}\left\|\left\{\left\|f\chi_{Q_{r,k}}\right\|_{L^{p(\cdot)}}\right\}_{k\in\mathbb{Z}^n}\right\|_{\ell^{q}}
&=
\left\|r^{\frac{n}{\alpha}-N_{p,r}-\frac{n}{q}}\sum_{k\in\mathbb{Z}^n}\left\|f\chi_{Q_{r,k}}\right\|_{L^{p(\cdot)}}\chi_{Q_{r,k}}\right\|_{L^{q}}\\
&\lesssim\left\|r^{\frac{n}{\alpha}-\frac{n}{p(\cdot)}-\frac{n}{q}}\left\|f\chi_{B(\cdot,2\sqrt{n}r)}\right\|_{L^{p(\cdot)}}\right\|_{L^{q}}.
\end{align*}\par
The proof is completed.
\end{proof}

By Lemma \ref{l3.2}, the proof of Proposition \ref{p3.1} is easy.
\begin{proof}[Proof of Proposition \ref{p3.1}]
According to the Lemma \ref{l3.2}, we obtain that
$$\sup_{r>0}\left\|r^{\frac{n}{\alpha}-N_{p,r}}\left\{\|f\chi_{Q_{r,k}}\|_{L^{p(\cdot)}}\right\}_{k\in\mathbb{Z}^n}\right\|_{\ell^{q}}\sim \sup_{r>0}\left\||B(\cdot,r)|^{\frac{1}{\alpha}-\frac{1}{p(\cdot)}-\frac{1}{q}}\|f\chi_{B(\cdot,r)}\|_{L^{p(\cdot)}}\right\|_{L^{q}},$$
and if $\rm r=1$, we have
$$\left\|\left\{\|f\chi_{Q_{1,k}}\|_{L^{p(\cdot)}}\right\}_{k\in\mathbb{Z}^n}\right\|_{\ell^{q}}\sim \left\|\|f\chi_{B(\cdot,1)}\|_{L^{p(\cdot)}}\right\|_{L^{q}}.$$

Thus, we prove proposition \ref{p3.1}.
\end{proof}\par

According to Proposition \ref{p3.1}, we give the definition of the pre-dual of variable Fofana's spaces $(L^{p(\cdot)},L^{q})^\alpha(\mathbb{R}^n)$.
\begin{definition}\label{d3.1}
Let $p(\cdot)\in \mathcal{P}(\R^n), 1\leq q, \alpha\leq\infty$ and $p(\cdot)\leq \alpha \leq q$. The space $\mathcal{H}(p(\cdot)',q',\alpha')$ is defined as the set of all elements of $L^{p(\cdot)}_{loc}(\R^n)$ for which there exist a sequence $\{(c_j,r_j,f_j)\}_{j\ge 1}$ of elements of $\mathbb{C}\times(0,\infty)\times(L^{p(\cdot)'},\ell^{q'})(\R^n)$ such that
\begin{equation}\label{3}
f:=\sum_{j\ge 1}c_j St_{r_j}^{(\alpha')}(f_j)\text{ in the sense of }L^{p(\cdot)}_{loc}(\R^n);
\end{equation}
\begin{equation}\label{4}
\|f_j\|_{p(\cdot)',q'}\le 1,j\ge 1;
\end{equation}
\begin{equation}\label{5}
\sum_{j\geq 1}|c_j|<\infty.
\end{equation}
\end{definition}

We will always refer to any sequence $\{(c_j,r_j,f_j)\}_{j\ge 1}$ of elements of $\mathbb{C}\times(0,\infty)\times(L^{p(\cdot)'},\ell^{q'})(\R^n)$ satisfying (\ref{3})-(\ref{5}) as block decomposition of $f$. For any element $f$ of $\mathcal{H}(p(\cdot)',q',\alpha')$, we set
$$\|f\|_{\mathcal{H}(p(\cdot)',q',\alpha')}:=\inf\left\{\sum_{j\ge 1}|c_j|:f:=\sum_{j\ge 1}c_j St_{r_j}^{(\alpha')}f_j\right\},$$
where the infimum is taken over all block decomposition of $f$.\par
Now, we discuss the properties of the dilation operator $St_{r}^{(\alpha)}:f\mapsto r^{-\frac{n}{\alpha}}f(r^{-1}\cdot)$ for $0<\alpha<\infty$ and $0<r<\infty$. By direct computation, we have the following properties.
\begin{proposition}\label{p3.2}
Let $f\in L^{p(\cdot)}_{loc}(\R^n)$, $0<\alpha<\infty$, and $0<r<\infty$.
\begin{itemize}
\item[(\romannumeral1)] $St_{r}^{(\alpha)}$ maps $L^{p(\cdot)}_{loc}(\R^n)$ into itself.
\item[(\romannumeral2)] $f=St_{1}^{(\alpha)}(f)$.
\item[(\romannumeral3)] $St_{r_1}^{(\alpha)}\circ St_{r_2}^{(\alpha)}=St_{r_2}^{(\alpha)}St_{r_2}^{(\alpha)}=St_{r_1r_2}^{(\alpha)}$.
\item[(\romannumeral4)] $\sup_{r>0}\|St_{r}^{(\alpha)}(f)\|_{p(\cdot),q}=\|f\|_{p(\cdot),q,\alpha}$, where $p(\cdot)\in \mathcal{P}(\R^n), 1\leq q, \alpha\leq\infty$ and $p(\cdot)\leq \alpha \leq q$.
\end{itemize}
\end{proposition}

Proposition \ref{p3.2} and Definition \ref{d3.1} prove the following result.

\begin{proposition}\label{p3.3}
Let $p(\cdot)\in \mathcal{P}(\R^n)$, $1\le q, \alpha\le\infty$, and $p(\cdot)\leq \alpha \leq q$. $(L^{p(\cdot)'},\ell^{q'})(\R^n)$ is a dense subspace of $\mathcal{H}(p(\cdot)',q',\alpha')$.
\end{proposition}

\begin{proof}
First we verify that $(L^{p(\cdot)'},\ell^{q'})(\R^n)$ is continuously embedded into $\mathcal{H}(p(\cdot)',q',\alpha')$. For any $f\in(L^{p(\cdot)'},\ell^{q'})(\R^n)$, we have
\begin{equation}\label{6}
f=\|f\|_{p(\cdot)',q'}St_1^{\alpha}(\|f\|^{-1}_{p(\cdot)',q'}f)
\end{equation}
and
$$\left\|\|f\|^{-1}_{p(\cdot)',q'}f\right\|_{p(\cdot)',q'}=1.$$\par
Thus, $f\in\mathcal{H}(p(\cdot)',q',\alpha')$ and satisfies
\begin{equation}\label{7}
\|f\|_{\mathcal{H}(p(\cdot)',q',\alpha')}\le\|f\|_{p(\cdot)',q'}.
\end{equation}

Let us show the denseness of $(L^{p(\cdot)'},\ell^{q'})(\R^n)$ in $\mathcal{H}(p(\cdot)',q',\alpha')$. It is clear that if $\{(c_j,r_j,f_j)\}_{j\ge 1}$ is a block decomposition of $f\in\mathcal{H}(p(\cdot)',q',\alpha')$, then
$\left\{\sum_{j=1}^Jc_j St_{r_j}^{(\alpha')}(f_j)\right\}_{J\ge 1}$ is a sequence of elements of $(L^{p(\cdot)'},\ell^{q'})(\R^n)$ and
$$\left\|f-\sum_{j=1}^{J}c_jSt_{r_j}^{(\alpha')}(f_j)\right\|_{\mathcal{H}(p(\cdot)',q',\alpha')} =\left\|\sum_{j=J+1}^{\infty}c_jSt_{r_j}^{(\alpha')}(f_j)\right\|_{\mathcal{H}(p(\cdot)',q',\alpha')}\le\sum_{j=J+1}^{\infty}|c_j|\rightarrow 0~~(J\rightarrow\infty).$$\par
Thus, $(L^{p(\cdot)'},\ell^{q'})(\R^n)$ is a dense subspace of $\mathcal{H}(p(\cdot)',q',\alpha')$.
\end{proof}

\begin{theorem}\label{t3.1}
\begin{itemize}
\item[(\romannumeral1)] Let $p(\cdot)\in \mathcal{P}(\R^n), 1\leq q, \alpha\leq\infty$ and $p(\cdot)\leq \alpha \leq q$. If $g\in(L^{p(\cdot)},\ell^{q})^\alpha(\R^n)$ and $f\in\mathcal{H}(p(\cdot)',q',\alpha')$, we obtain $fg\in L^1(\R^n)$ and
\begin{equation}\label{8}
\left|\int_{\R^n}f(x)g(x)dx\right|\le\|g\|_{{p(\cdot)},{q},\alpha} \|f\|_{\mathcal{H}(p(\cdot)',q',\alpha')}.
\end{equation}
\item[(\romannumeral2)] Let $p(\cdot)\in \mathcal{P}(\R^n), 1\leq q, \alpha\leq\infty$ and $p(\cdot)\leq \alpha \leq q$. The operator $T:g\mapsto T_g$ defined as
$$<T_g,f>=\int_{\R^n}f(x)g(x)dx,~~g\in(L^{p(\cdot)},\ell^{q})^{\alpha}(\R^n)\text{ and }f\in\mathcal{H}{(p(\cdot)',q',\alpha')}$$
is an isomestric isomorphism of $(L^{p(\cdot)},\ell^{q})^{\alpha}(\R^n)$ into $\mathcal{H}{(p(\cdot)',q',\alpha')}^{*}$.
\end{itemize}
\end{theorem}

Next, we will prove Theorem \ref{t3.1}, whose ideal comes from [14]. Before that, the dual of variable amalgam spaces $(L^{p(\cdot)},\ell^{q})(\R^n)$ will given as follows.
\begin{lemma}\label{4.1}
\begin{itemize}
\item[(\romannumeral1)] Let $p(\cdot)\in\mathbb{R}^n$, $1\le q\le\infty$. For $r\in(0,\infty)$, if $f\in(L^{p(\cdot)},\ell^{q})(\R^n)$ and $g\in(L^{p(\cdot)'},\ell^{q'})(\R^n)$, then $fg\in L^1(\R^n)$ and
\begin{equation}\label{9}
\|fg\|_1\lesssim{\|f\|_{p(\cdot),q}}\cdot{\|g\|_{p(\cdot)',q'}},~~f,g\in L_{loc}^{p(\cdot)}(\R^n).
\end{equation}
\item[(\romannumeral2)] Let $p(\cdot)\in \mathcal{P}(\R^n)$, $1\le q<\infty$. The dual of variable amalgam spaces $(L^{p(\cdot)},\ell^{q})(\R^n)$ is $(L^{p(\cdot)'},\ell^{q'})(\R^n)$.
\end{itemize}
\end{lemma}

\begin{proof}
For $0<r<\infty$, by H\"older's inequality, we have
$$\|fg\|_1\le{\|f\|_{p(\cdot),q}}\cdot{\|g\|_{p(\cdot)',q'}},~~f,g\in L_{loc}^{p(\cdot)}(\R^n).$$\par
According to ([11]Theorem 2) and ([2]Theorem 2.6), it immediate to deduce that the dual of $(L^{p(\cdot)},\ell^{q})(\R^n)$ is $(L^{p(\cdot)'},\ell^{q'})(\R^n)$. Let $q=\{q_n\}_{n\in\mathbb{N}}$, if the dual of $(L^{p(\cdot)},\ell^{\bar{q}})(\R^n)$ is $(L^{p(\cdot)'},\ell^{\bar{q}'})(\R^n)$ with $\bar{q}=(q_1,q_2,\cdots,q_{n-1})$, then
\begin{align*}
 (L^{p(\cdot)},\ell^{q})^{*}=\left(\prod(L^{p(\cdot)},\ell^{\bar{q}}),\ell^{q_n}\right)^{*} &=\left(\prod(L^{p(\cdot)},\ell^{\bar{q}})^{*},\left(\ell^{q_n}\right)^{*}\right) \\
 &=\left(\prod(L^{p(\cdot)'},\ell^{\bar{q}'}),\ell^{q'_n}\right)=(L^{p(\cdot)'},\ell^{q'}).
\end{align*}\par
Hence, $(L^{p(\cdot)'},\ell^{q'})(\R^{n})$ is isometrically isomorphic to the dual of $(L^{p(\cdot)},\ell^{q})(\R^{n})$. There is an unique element $\phi(T)$ of $(L^{p'(\cdot)},\ell^{q'})(\R^{n})$ such that
$$T(f)=\int_{\R^n}f(x)\phi(T)(x)dx,~~f\in(L^{p(\cdot)},\ell^{q})(\R^{n})$$
and
\begin{equation}\label{10}
\|\phi(T)\|_{p(\cdot)',q'}=\|T\|,
\end{equation}
where $\|T\|:=\sup\left\{|T(f)|:f\in L_{loc}^{p(\cdot)}(\R^n) \text{ and }\|f\|_{p(\cdot),q}\le 1\right\}$.
\end{proof}

Now, let us prove the main theorem in this section.
\begin{proof}[Proof of Theorem \ref{t3.1}]
Let us to prove (\romannumeral1). Let $\{(c_j,r_j,f_j)\}_{j\ge 1}$ be block decomposition of $f$. By Proposition \ref{4.1} and (\ref{9}), we have for any $j\ge 1$,
\begin{align*}
\left|\int_{\R^n} St_{r_j}^{(\alpha')}(f_j)(x)g(x)dx\right|&=\left|\int_{\R^n} St_{r^{-1}_j}^{(\alpha)}(g)(x)f_j(x)dx\right|\\
&\le\int_{\R^n}\left|St_{r^{-1}_j}^{(\alpha)}(g)(x)f_j(x)\right|dx\\
&\le\|f_j\|_{p(\cdot)',q'}\left\|St_{r^{-1}_j}^{(\alpha)}(g)\right\|_{p(\cdot),q}\\
&\le\left\|St_{r^{-1}_j}^{(\alpha)}(g)\right\|_{p(\cdot),q}\\
&\le\|g\|_{p(\cdot),q,\alpha}.
\end{align*}\par
Therefore we have
$$\sum_{j\ge 1}\int_{\R^n}\left|c_jSt_{r_j}^{(\alpha')}(f_j)(x)g(x)\right|dx\le\|g\|_{p(\cdot),q,\alpha}\sum_{j\ge 1}|c_j|.$$\par
This implies that $fg=g\sum_{j\ge 1}c_jSt_{r_j}^{(\alpha')}(f_j)\in L^1(\R^n)$ and
$$\left|\int_{\R^n}f(x)g(x)dx\right|\le\int_{\R^n}|f(x)g(x)|dx\le\|g\|_{p(\cdot),q,\alpha}\sum_{j\ge 1}|c_j|,$$
taking the infimum with respect to all block decompositions of $f$, we get
$$\left|\int_{\R^n}f(x)g(x)dx\right|\le\|g\|_{p(\cdot),q,\alpha}\|f\|_{\mathcal{H}(p(\cdot)',q',\alpha')}.$$\par

Now, Let us prove (\romannumeral2). By the (\romannumeral1), we have $T_g\in\mathcal{H}(p(\cdot)',q',\alpha')^{*}$.\par
For any $a_1,a_2\in \R,~g_1,g_2\in(L^{p(\cdot)},\ell^{q})^{\alpha}(\R^n)$, it's easy to see
$$T(a_1g_1+a_2g_2)=a_1T_{g_1}+a_2T_{g_2}$$
and
$$\|T_g\|=\sup_{\|f\|_{\mathcal{H}(p(\cdot)',q',\alpha')}\le 1}|T_g(f)|\le\|g\|_{p(\cdot),q,\alpha},$$
that is, $T$ is linear and bounded mapping from $(L^{p(\cdot)},\ell^{q})^{\alpha}(\R^n)$ into $\mathcal{H}{(p(\cdot)',p(\cdot)',\alpha')}^{*}$ satisfying $\|T\|\le 1$. For any $g_1,g_2\in(L^{p(\cdot)},\ell^{q})^{\alpha}(\R^n)\subset(L^{p(\cdot)},\ell^{q})(\R^n)$, if $T_{g_1}=T_{g_2}$, then for any $f\in(L^{p(\cdot)'},\ell^{q'})(\R^n)\subset\mathcal{H}(p(\cdot)',q',\alpha')$, we have $T_{g_1}(f)=T_{g_2}(f).$\par
Thus, $g_1=g_2$, that is , $T$ is injective.

Now, we will prove that $T$ is a surjection and $\|g\|_{p(\cdot),q,\alpha}\le\|T_g\|~(\text{or }\|T\|\ge 1)$. Let $T$ be an element of $\mathcal{H}(p(\cdot)',q',\alpha')^{*}$. From Proposition 3.3, it follows that the restriction $T_0$ of $T$ to $(L^{p(\cdot)'},\ell^{q'})(\R^n)$ belong to $\mathcal{H}(p(\cdot)',q',\alpha')^{*}$. Furthermore, we have $\frac{1}{p(\cdot)'}\leq\frac{1}{\alpha'}\le\frac{1}{q'}.$\par
There is an element $g$ of $(L^{p(\cdot)},\ell^{q})(\R^n)$, such that for any $f\in(L^{p(\cdot)'},\ell^{q'})(\R^n)$
\begin{equation}\label{11}
T(f)=T_0(f)=\int_{\R^n}f(x)g(x)dx.
\end{equation}

Hence, for $f\in(L^{p(\cdot)'},\ell^{q'})(\R^n)$ and $r>0$, we have
$$\int_{\R^n}St^{(\alpha)}_{r}(g)(x)f(x)dx=\int_{\R^n}g(x)St^{(\alpha')}_{r^{-1}}(f)(x)dx=T\left[St^{(\alpha')}_{r^{-1}}(f)\right].$$\par
By the assumption $T\in\mathcal{H}(p(\cdot)',q',\alpha')^{*}$, we have $St^{(\alpha')}_{r^{-1}}(f)\in\mathcal{H}(p(\cdot)',q',\alpha')$ and using (\ref{7})
$$\left|\int_{\R^n}St^{(\alpha)}_{r}(g)(x)f(x)dx\right| \le\|T\|\cdot\|St^{(\alpha')}_{\rho^{-1}}(f)\|_{\mathcal{H}(p(\cdot)',q',\alpha')} \le\|T\|\cdot\|f\|_{p(\cdot)',q'}.$$\par
Due to (\ref{10}), it follows that
$St^{(\alpha)}_{r}(g)\in(L^{p(\cdot)},\ell^{q})(\R^n)$ and $\|St^{(\alpha)}_{r}(g)\|_{p(\cdot),q}\le\|T\|.$\par
Therefore, for any $g\in(L^{p(\cdot)},\ell^{q})(\R^n)$, by Proposition \ref{p3.2},
$$\|g\|_{p(\cdot),q,\alpha}\le\|T\|.$$\par
According to (\ref{11}) and Proposition \ref{p3.3}, we get
$$T(f)=\int_{\R^n}f(x)g(x)dx,~~f\in\mathcal{H}(p(\cdot)',q',\alpha').$$\par
Thus, $T$ is a surjection and $\|g\|_{p(\cdot),q,\alpha}\le\|T\|$.

\end{proof}

\begin{proposition}\label{p3.4}
Let $p(\cdot)\in \mathcal{P}(\R^n), 1\leq q, \alpha\leq\infty$ and $p(\cdot)\leq \alpha \leq q$ and $\chi_{B(x_0,r_0)}$ is a characteristic function on $B(x_0,r_0)$. Then we have
$$\|\chi_{B(x_0,r_0)}\|_{(L^{p(\cdot)},L^{q})^\alpha}\lesssim r_0^{n/\alpha+C_p}\text{ and }\|\chi_{B(x_0,r_0)}\|_{\mathcal{H}(p(\cdot)',q',\alpha')}\lesssim r_0^{n/\alpha'},$$
where
\begin{eqnarray*}
C_p=
\begin{cases}
\frac{n}{p_-}-\frac{n}{p^+}  ~~~~~~r>r_0\geq1,\\
0~~~~~~~~~~~~~~~r>1>r_0~~or~~r<r_0,\\
\frac{n}{p^+}-\frac{n}{p_-}  ~~~~~1\geq r>r_0.
\end{cases}
\end{eqnarray*}
\end{proposition}

\begin{proof}
It is obviously that
\begin{align*}
\|\chi_{B(x_0,r_0)}\|_{(L^{p(\cdot)},L^{q})^\alpha} \sim\sup_{r>0}
\left\|r^{\frac{n}{\alpha}-\frac{n}{p(\cdot)}-\frac{n}{q}} \|\chi_{B(x_{0},r_0)}\chi_{B(x,r)}\|_{L^{p(\cdot)}} \right\|_{L^{q}}.
\end{align*}\par

If $r>r_0$ and $p(\cdot)\in \mathcal{P}(\R^n)$.
When $r>r_0\geq1$, then by $\frac{n}{\alpha}-\frac{n}{p(\cdot)}\le 0$,
\begin{align*}
\sup_{r>r_0}
\left\|r^{\frac{n}{\alpha}-\frac{n}{p(\cdot)}-\frac{n}{q}}\|\chi_{B(x_{0},r_0)}\chi_{B(\cdot,r)}\|_{L^{p(\cdot)}} \right\|_{L^{q}}
&\le\sup_{r>r_0}r^{\frac{n}{\alpha}-\frac{n}{p^+}-\frac{n}{q}}
\left\| \|\chi_{B(x_{0},r_0)}\|_{L^{p(\cdot)}}\cdot\chi_{B(x_0,r+r_0)}\right\|_{L^{q}}\\
&\lesssim r_0^{\frac{n}{p_-}}\sup_{r>r_0}r^{\frac{n}{\alpha}-\frac{n}{p^+}}
(1+\frac{r_0}{r})^{\frac{n}{q}}\\
&\lesssim r_0^{\frac{n}{\alpha}+C_p},
\end{align*}
when $r>1>r_0$, then by $\frac{n}{\alpha}-\frac{n}{p(\cdot)}\le 0$,
\begin{align*}
\sup_{r>r_0}
\left\|r^{\frac{n}{\alpha}-\frac{n}{p(\cdot)}-\frac{n}{q}}\|\chi_{B(x_{0},r_0)}\chi_{B(\cdot,r)}\|_{L^{p(\cdot)}} \right\|_{L^{q}}
&\le\sup_{r>r_0}r^{\frac{n}{\alpha}-\frac{n}{p^+}-\frac{n}{q}}
\left\| \|\chi_{B(x_{0},r_0)}\|_{L^{p(\cdot)}}\cdot\chi_{B(x_0,r+r_0)}\right\|_{L^{q}}\\
&\lesssim r_0^{\frac{n}{p^+}}\sup_{r>r_0}r^{\frac{n}{\alpha}-\frac{n}{p^+}}
(1+\frac{r_0}{r})^{\frac{n}{q}}\\
&\lesssim r_0^{\frac{n}{\alpha}+C_p},
\end{align*}
when $1\geq r>r_0\geq1$, then by $\frac{n}{\alpha}-\frac{n}{p(\cdot)}\le 0$,
\begin{align*}
\sup_{r>r_0}
\left\|r^{\frac{n}{\alpha}-\frac{n}{p(\cdot)}-\frac{n}{q}}\|\chi_{B(x_{0},r_0)}\chi_{B(\cdot,r)}\|_{L^{p(\cdot)}} \right\|_{L^{q}}
&\le\sup_{r>r_0}r^{\frac{n}{\alpha}-\frac{n}{p_-}-\frac{n}{q}}
\left\| \|\chi_{B(x_{0},r_0)}\|_{L^{p(\cdot)}}\cdot\chi_{B(x_0,r+r_0)}\right\|_{L^{q}}\\
&\lesssim r_0^{\frac{1}{p^+}}\sup_{r>r_0}r^{\frac{n}{\alpha}-\frac{n}{p_-}}
(1+\frac{r_0}{r})^{\frac{n}{q}}\\
&\lesssim r_0^{\frac{n}{\alpha}+C_p}.
\end{align*}\par
For $r\le r_0$ and $p(\cdot)\in \mathcal{P}(\R^n)$, by $\frac{1}{\alpha}-\frac{1}{q}\ge 0$ we have
\begin{align*}
&\sup_{r\le r_0}r^{\frac{n}{\alpha}-\frac{n}{q}}
\left\|\frac{1}{\left\|\chi_{B(\cdot,r)}\right\|_{L^{p(\cdot)}}} \left\| \chi_{B(x_{0},r_0)}\chi_{B(\cdot,r)}\right\|_{L^{p(\cdot)}} \right\|_{L^{q}}\\
&\le\sup_{r\le r_0}r^{\frac{n}{\alpha}-\frac{n}{q}}
\left\|\frac{1}{\left\|\chi_{B(\cdot,r)}\right\|_{L^{p(\cdot)}}}\cdot\left\|\chi_{B(\cdot,r)}\right\|_{L^{p(\cdot)}}\chi_{B(x_0,r+r_0)}\right\|_{L^{q}}\\
&=\sup_{0<r\le r_0}r^{\frac{n}{\alpha}}(1+\frac{r_0}{r})^{\frac{n}{q}}\\
&\lesssim r_0^{\frac{n}{\alpha}+C_p}.
\end{align*}\par

Thus, $\|\chi_{B(x_0,r_0)}\|_{(L^{p(\cdot)},L^{q})^{\alpha}}\lesssim r_0^{n/\alpha+C_p}.$

Next, we show that $\|\chi_{B(x_0,r_0)}\|_{\mathcal{H}(p(\cdot),q',\alpha')}\lesssim r_0^{n/\alpha'}$. First, by the similar argument with dilation operator of (\ref{6}), let
$$\chi_{B(x_0,r_0)}=r^{\frac{n}{\alpha'}}\|\chi_{B(x_0/r,r_0/r)}\|_{p(\cdot)',q'}\cdot St_r^{\alpha'}(\|\chi_{B(x_0/r,r_0/r)}\|_{p(\cdot)',q'}^{-1}\chi_{B(x_0/r,r_0/r)}).$$\par
It is obvious that
$$\left\|\|\chi_{B(x_0/r,r_0/r)}\|_{p(\cdot)',q'}^{-1}\chi_{B(x_0/r,r_0/r)}\right\|_{p(\cdot)',q'}\le 1.$$\par
From Definition \ref{d3.1} and Proposition \ref{p3.1},
\begin{align*}
\|\chi_{B(x_0,r_0)}\|_{\mathcal{H}(p(\cdot)',q',\alpha')} &\le\sup_{r>0}r^{\frac{n}{\alpha'}}\|\chi_{B(x_0/r,r_0/r)}\|_{p(\cdot)',q'}
\lesssim \sup_{r>0}r^{\frac{n}{\alpha'}}\left\|\|\chi_{B(x_0/r,r_0/r)}\chi_{B(\cdot,1)}\|_{L^{p(\cdot)'}} \right\|_{L^{q'}}.
\end{align*}\par
Using the same argument of the proof of $\|\chi_{B(x_0,r_0)}\|_{(L^{\vec{p}},L^{\vec{s}})^\alpha}\lesssim r_0^{n/\alpha}$ with $r_0/r>1$ and $r_0/r\le 1$, we have
\begin{align*}
\|\chi_{B(x_0,r_0)}\|_{\mathcal{H}(p(\cdot)',q',\alpha')} \lesssim\sup_{r>0}r^{\frac{n}{\alpha'}}\left\|\|\chi_{B(x_0/r,r_0/r)}\chi_{B(\cdot,1)}\|_{L^{p(\cdot)'}} \right\|_{L^{q'}}\lesssim r_0^{n/\alpha'}.
\end{align*}\par
The proof is completed.
\end{proof}

\section{\hspace{-0.6cm}{\bf }~~The characterization of fractional integral operators and commutators on variable Fofana's spaces\label{s4}}

In this section, we need the following lemmas about the $BMO(\R^n)$ function.
\begin{lemma}\label{l4.1}
Let $b\in BMO(\R^n)$. For any ball $B$ in $\R^n$ and for any positive integer $j\in \mathbb{Z}^{+}$,
$$|b_{2^{j+1}B}-b_B|\lesssim(j+1)\|b\|_{*}.$$
\end{lemma}
\begin{proof}[Proof of Lemma \ref{l4.1}]
\begin{align*}
|b_{2^{j+1}B}-b_B|&\leq \sum_{k=0}^j|b_{2^{k+1}B}-b_{2^{k}B}|\\
                  &\lesssim\sum_{k=0}^j\frac{1}{|2^{k}B|}\int_{2^{k}B}|b(y)-b_{2^{k+1}B}|dy\\
                  &\lesssim\sum_{k=0}^j\frac{1}{|2^{k}B|}\|(b-b_{2^{k+1}B})\chi_{2^{k+1}B}\|_{L^{p(\cdot)}(\mathbb{R}^n)}\|\chi_{2^{k+1}B}\|_{L^{p'(\cdot)}(\mathbb{R}^n)}\\
                  &\lesssim\|b\|_{*}\sum_{k=0}^j\frac{1}{|2^{k}B|}\|\chi_{2^{k+1}B}\|_{L^{p(\cdot)}(\mathbb{R}^n)}\|\chi_{2^{k+1}B}\|_{L^{p'(\cdot)}(\mathbb{R}^n)}\\
                 &\lesssim(j+1)\|b\|_*.
\end{align*}\par
\end{proof}

\begin{lemma}[see\cite{17} Lemma 3]\label{l4.2}
Let $p(\cdot)\in\mathcal{B}(\mathbb{R}^n)$, k be a positive integer and B be a ball in $\mathbb{R}^n$. Then we have that for all $b\in BMO(\mathbb{R}^n)$ and all $j,~i\in \mathbb{Z}$ with $j>i$,
\begin{align*}
\frac{1}{C}\|b\|^k_*\leq\sup_{B}\frac{1}{\|\chi_B\|_{L^{p(\cdot)}(\mathbb{R}^n)}}\|(b-b_B)^k\chi_B\|_{L^{p(\cdot)}(\mathbb{R}^n)}\leq C\|b\|^k_*,\\
\|(b-b_{B_i})^k\chi_{B_j}\|_{L^{p(\cdot)}(\mathbb{R}^n)}\leq C(j-i)^k\|b\|^k_*\|\chi_{B_j}\|_{L^{p(\cdot)}(\mathbb{R}^n)},
\end{align*}
where $B_i=\{x\in \mathbb{R}^n : |x|\leq2^i\}$ and $B_j=\{x\in \mathbb{R}^n : |x|\leq2^j\}$.
\end{lemma}\par

\begin{lemma}[see\cite{19} ]\label{l4.3}
Given an open set $\Omega\in \R^n$ and $\alpha$, $0 <\alpha<n$, $p(\cdot)\in \mathcal{P}(\Omega)$ satisfies conditions (1), (2) in Lemma1.1. Define $q(\cdot):\Omega \rightarrow [1,\infty)$ by $\frac{1}{p(\cdot)}-\frac{1}{q(\cdot)}= \frac{\alpha}{n}$, then the
fractional integral operator $I_\alpha$ is bounded from $L^p(\cdot)(\Omega)$ to $L^q(\cdot)(\Omega)$.
\end{lemma}\par

\begin{lemma}[see\cite{21}]\label{l4.4}
Suppose that $p_1(\cdot)\in \mathcal{P}(\R^n)$ satisfies conditions (1), (2) in Lemma1.1. $0 <\alpha<\frac{n}{{p_1}^+}$. Define the variable exponent $p_2(\cdot)$ by $\frac{1}{p_1(\cdot)}-\frac{1}{p_2(\cdot)}= \frac{\alpha}{n}$, then we have
$$ \| [b,I_\alpha]f\|_{L^{p_2(\cdot)}(\R^n)}\lesssim\|b\|_{*}\|f\|_{L^{p_1(\cdot)}(\R^n)},$$
for all $f \in L^{p_1(\cdot)(\R^n)}$ and $b \in \rm BMO(\R^n)$.
\end{lemma}\par

\begin{theorem}\label{t4.1}
Let $p_1(\cdot), p_2(\cdot)\in \mathcal{P}(\R^n)$ satisfies conditions (1), (2) in Lemma1.1, $0<\gamma<n/{p_1^+}$, $1\leq q, \alpha\leq\infty$ and $p_1(\cdot)<\alpha<q<\infty$, $p_2(\cdot)<\beta<q<\infty$. Assume that  $\frac{n}{p_1(\cdot)}-\frac{n}{p_2(\cdot)}=\gamma$. Then the fractional integral operators $I_\gamma$ are bounded from $(L^{p_1(\cdot)},L^{q})^{\alpha}(\R^n)$ to $(L^{p_2(\cdot)},L^{q})^{\beta}(\R^n)$ if and only if $$\gamma=\frac{n}{\alpha}-\frac{n}{\beta}.$$
\end{theorem}
\begin{remark}\label{r4.1}
In fact, the condition $\gamma=\frac{n}{\alpha}-\frac{n}{\beta}$ is necessary for the boundedness of fractional integral operators $I_\gamma$.
Let $\delta_tf(x)=f(tx)$, where $(t>0)$. Then,
$$I_{\gamma}(\delta_t f)=t^{-\gamma}\delta_t I_{\gamma}(f).$$
$$\|\delta_{t^{-1}}f\|_{_{(L^{p_2(\cdot)},L^{q})^{\beta}}} =t^{\frac{n}{\beta}-\frac{1}{q}}\|f\|_{_{(L^{p_2(\cdot)},L^{q})^{\beta}}}.$$
$$\|\delta_{t}f\|_{_{(L^{p_1(\cdot)},L^{s})^{\alpha}}} =t^{-\frac{n}{\alpha}+\frac{1}{s}}\|f\|_{_{(L^{p_1(\cdot)},L^{s})^{\alpha}}}.$$
Thus, by the boundedness of $I_\gamma$ from $(L^{p_1(\cdot)},L^{s})^{\alpha}(\R^n)$ to $(L^{p_2(\cdot)},L^{q})^{\beta}(\R^n)$.\par
\begin{align*}
\|I_{\gamma}f\|_{(L^{p_2(\cdot)},L^{q})^{\beta}}
&=t^{\gamma}\|\delta_{t^{-1}}I_{\gamma}(\delta_tf)\|_{(L^{p_2(\cdot)},L^{q})^{\beta}}\\
&=t^{\gamma+\frac{n}{\beta}-\frac{1}{q}}\|I_{\gamma}(\delta_tf)\|_{(L^{p_2(\cdot)},L^{q})^{\beta}}\\
&\lesssim t^{\gamma+\frac{n}{\beta}-\frac{1}{q}}\|\delta_tf\|_{(L^{p_1(\cdot)},L^{q})^{\alpha}}\\
&=t^{\gamma+\frac{n}{\beta}-\frac{1}{q}-\frac{n}{\alpha}+\frac{1}{q}}
\|f\|_{(L^{p_1(\cdot)},L^{s})^{\alpha}}.
\end{align*}\par
Thus, $\gamma=\frac{n}{\alpha}-\frac{n}{\beta}$.
\end{remark}

\begin{proof}[Proof of Theorem \ref{t4.1}]
By Remark \ref{r4.1}, we only need to prove the boundedness of $I_\gamma$ on variable Fofana's spaces if $\gamma=\frac{n}{\alpha}-\frac{n}{\beta}$.
Let $p_{1}(\cdot), p_{2}(\cdot)\in\mathcal{P}(\R^n)$, $p_1(\cdot)<\alpha<q<\infty$, $p_2(\cdot)<\beta<q<\infty$, $0<\gamma<n/{p_1^+}$, $\frac{n}{p_1(\cdot)}-\frac{n}{p_2(\cdot)}=\gamma, \frac{n}{\alpha}-\frac{n}{\beta}=\gamma$ and $f\in (L^{p_1(\cdot)},L^q)^\alpha (\mathbb{R}^n)$
Fix $x\in\mathbb{R}^n$, $r>0$ and set $B=B(x, r)$, $2B= B(x,2r)$. Let $f_1=f\chi_{2B}$ and $f_2=f-f_1$. We write
 \begin{align*}
 |B|^{\frac{1}{\beta}-\frac{1}{p_2(\cdot)}-\frac{1}{q}}\|I_{\gamma} f\chi_{B} \|_{L^{p_2(\cdot)}(\mathbb{R}^n)}
 &\leq|B|^{\frac{1}{\beta}-\frac{1}{p_2(\cdot)}-\frac{1}{q}}\|I_{\gamma} f_1\chi_{B} \|_{L^{p_2(\cdot)}(\mathbb{R}^n)}
 +|B|^{\frac{1}{\beta}-\frac{1}{p_2(\cdot)}-\frac{1}{q}}\|I_{\gamma} f_2\chi_{B} \|_{L^{p_2(\cdot)}(\mathbb{R}^n)}\\
 &=:I_1+I_2.
\end{align*}\par
We first estimate $ I_1$. Since the fractional integral operators $I_{\gamma}$ are bounded from $L^{p_1(\cdot)}(\mathbb{R}^n)$ to $L^{p_2(\cdot)}(\mathbb{R}^n)$ and $\frac{n}{p_1(\cdot)}-\frac{n}{p_2(\cdot)}=\gamma, \frac{n}{\alpha}-\frac{n}{\beta}=\gamma$, we have $\frac{1}{\beta}-\frac{1}{p_2(\cdot)}-\frac{1}{q}=\frac{1}{\alpha}-\frac{1}{p_1(\cdot)}-\frac{1}{q}$ and  $-1<\frac{1}{\alpha}-\frac{1}{{p_1}_-}-\frac{1}{q}<\frac{1}{\alpha}-\frac{1}{p_1(\cdot)}-\frac{1}{q}<0$. Hence it suffices to note that
\begin{align*}
I_1&\leq|B|^{\frac{1}{\beta}-\frac{1}{p_2(\cdot)}-\frac{1}{q}}\|I_{\gamma} f_1\|_{L^{p_2(\cdot)}(\mathbb{R}^n)}\\
   &\lesssim|B|^{\frac{1}{\beta}-\frac{1}{p_2(\cdot)}-\frac{1}{q}}\|f\chi_{2B}\|_{L^{p_1(\cdot)}(\mathbb{R}^n)}\\
   &\lesssim|2B|^{\frac{1}{\alpha}-\frac{1}{p_1(\cdot)}-\frac{1}{q}}\|f\chi_{2B}\|_{L^{p_1(\cdot)}(\mathbb{R}^n)}
     \bigg(\frac{|B|}{|2B|}\bigg)^{\frac{1}{\alpha}-\frac{1}{{p_1}_-}-\frac{1}{q}}\\
   &\lesssim|2B|^{\frac{1}{\alpha}-\frac{1}{p_1(\cdot)}-\frac{1}{q}}\|f\chi_{2B}\|_{L^{p_1(\cdot)}(\mathbb{R}^n)}.
\end{align*}\par
Next, begin to estimate $I_2$. It is obvious that $|y - z| \approx |x - z|$, when $y \in B(x,r)$ and $z \in B(x,2r)^c$. Decompose $\mathbb{R}^n $ into a geometrically increasing sequence of concentric balls, using the generalized H{\"{o}}lder inequality, we have
\begin{align*}
|I_{\gamma}f_2(y)|&\lesssim\int_{(2B)^c}{\frac{|f(z)|}{|y-z|^{n-\gamma}}}dz\\
                          &\lesssim\sum_{j=1}^{\infty}\int_{2^{j+1}B\backslash 2^jB}{\frac{|f(z)|}{|x-z|^{n-\gamma}}}dz\\
                          &\lesssim\sum_{j=1}^{\infty}|2^{j+1}B|^{\frac{\gamma}{n}-1}\int_{2^{j+1}B}|f(z)|dz\\
                          &\lesssim\sum_{j=1}^{\infty}|2^{j+1}B|^{\frac{\gamma}{n}-1}\|f\chi_{2^{j+1}B}\|_{L^{p_1(\cdot)}(\mathbb{R}^n)}
                          \|\chi_{2^{j+1}B}\|_{L^{p_1'(\cdot)}(\mathbb{R}^n)}.
\end{align*}\par

From above estimate, by Lemma \ref{l1.3}, $\frac{1}{\beta}-\frac{1}{q}>0$ and $\frac{n}{p_1(\cdot)}-\frac{n}{p_2(\cdot)}=\gamma, \frac{n}{\alpha}-\frac{n}{\beta}=\gamma$, we have $\frac{1}{\beta}-\frac{1}{p_2(\cdot)}-\frac{1}{q}=\frac{1}{\alpha}-\frac{1}{p_1(\cdot)}-\frac{1}{q}$, it follows that
\begin{align*}
I_2&\lesssim|B|^{\frac{1}{\beta}-\frac{1}{p_2(\cdot)}-\frac{1}{q}}\sum_{j=1}^{\infty}|2^{j+1}B|^{\frac{\gamma}{n}-1}
         \|f\chi_{2^{j+1}B}\|_{L^{p_1(\cdot)}(\mathbb{R}^n)}
         \|\chi_{2^{j+1}B}\|_{L^{p_1'(\cdot)}(\mathbb{R}^n)}\|\chi_{B}\|_{L^{p_2(\cdot)}(\mathbb{R}^n)}\\
   &\lesssim\sum_{j=1}^{\infty}|2^{j+1}B|^{\frac{1}{\alpha}-\frac{1}{p_1(\cdot)}-\frac{1}{q}}
         \|f\chi_{2^{j+1}B}\|_{L^{p_1(\cdot)}(\mathbb{R}^n)}
         \frac{\|\chi_B\|_{L^{p_2(\cdot)}(\mathbb{R}^n)}}{\|\chi_{2^{j+1}B}\|_{L^{p_2(\cdot)}(\mathbb{R}^n)}}
         \frac{|B|^{\frac{1}{\beta}-\frac{1}{p_2(\cdot)}-\frac{1}{q}}}{|2^{j+1}B|^{\frac{1}{\alpha}-\frac{1}{p_1(\cdot)}-\frac{1}{q}}}\\
    &\lesssim\sum_{j=1}^{\infty}|2^{j+1}B|^{\frac{1}{\alpha}-\frac{1}{p_1(\cdot)}-\frac{1}{q}}
         \|f\chi_{2^{j+1}B}\|_{L^{p_1(\cdot)}(\mathbb{R}^n)}
        \frac{\|\chi_B\|_{L^{p_2(\cdot)}(\mathbb{R}^n)}}{\|\chi_{2^{j+1}B}\|_{L^{p_2(\cdot)}(\mathbb{R}^n)}}
        \frac{\|\chi_{2^{j+1}B}\|_{L^{p_2(\cdot)}(\mathbb{R}^n)}}{\|\chi_B\|_{L^{p_2(\cdot)}(\mathbb{R}^n)}}
         \times\bigg(\frac{|B|}{|2^{j+1}B|}\bigg)^{\frac{1}{\beta}-\frac{1}{q}}\\
   &\lesssim\sum_{j=1}^{\infty}|2^{j+1}B|^{\frac{1}{\alpha}-\frac{1}{p_1(\cdot)}-\frac{1}{q}}
         \|f\chi_{2^{j+1}B}\|_{L^{p_1(\cdot)}(\mathbb{R}^n)}
         \times\bigg(\frac{1}{2^{(j+1)n}}\bigg)^{\frac{1}{\beta}-\frac{1}{q}}.
\end{align*}\par

Taking the $L^q$-norm and using Minkowski's inequality, we get

\begin{align*}
&\left\||B|^{\frac{1}{\beta}-\frac{1}{p_2(\cdot)}-\frac{1}{q}}\|I_\gamma f\chi_{B}\|_{L^{p_2(\cdot)}(\mathbb{R}^n)}\right\|_{L^q(\mathbb{R}^n)}\\
&\leq \|I_1\|_{L^q(\mathbb{R}^n)}+\|I_2\|_{L^q(\mathbb{R}^n)}\\
&\lesssim\bigg(\left\||2B|^{\frac{1}{\alpha}-\frac{1}{p_1(\cdot)}-\frac{1}{q}}\|f\chi_{2B}\|_{L^{p_1(\cdot)}(\mathbb{R}^n)}\right\|_{L^q(\mathbb{R}^n)}\\
&~~~~+\sum_{j=1}^{\infty}\left\||2^{j+1}B|^{\frac{1}{\alpha}-\frac{1}{p_1(\cdot)}-\frac{1}{q}}\|f\chi_{2^{j+1}B}\|_{L^{p_1(\cdot)}(\mathbb{R}^n)}\right\|_{L^q(\mathbb{R}^n)}\bigg(\frac{1}{2^{(j+1)n}}\bigg)^{\frac{1}{\beta}-\frac{1}{q}}\bigg).
\end{align*}\par
Thus, by taking the supremum over all $ r > 0$, the proof is completed.
\end{proof}

\begin{theorem}\label{t4.2}
Let $p_1(\cdot), p_2(\cdot)\in \mathcal{P}(\R^n)$ satisfies conditions (1) and (2) in Lemma1.1, $0<\gamma<n/{p_1^+}$, $1\leq q, \alpha\leq\infty$ and $p_1(\cdot)<\alpha<q<\infty$, $p_2(\cdot)<\beta<q<\infty$. Assume that $\frac{n}{p_1(\cdot)}-\frac{n}{p_2(\cdot)}=\gamma$ and $\frac{n}{\alpha}-\frac{n}{\beta}=\gamma$, then the following statements are equivalent:
\begin{itemize}
\item[(\romannumeral1)] $b\in BMO(\R^n)$;
\item[(\romannumeral2)] The linear commutators $[b,I_\gamma]$ are bounded from $(L^{p_1(\cdot)},L^{q})^{\alpha}(\R^n)$ to $(L^{p_2(\cdot)},L^{q})^{\beta}(\R^n)$.
\end{itemize}
\end{theorem}
\begin{proof}[Proof of Theorem \ref{t4.2}]
Let $p_{1}(\cdot), p_{2}(\cdot)\in\mathcal{P}(\R^n)$, $p_1(\cdot)<\alpha<q<\infty, p_2(\cdot)<\beta<q<\infty$, $0<\gamma<n/{p_1^+}$, $\frac{n}{p_1(\cdot)}-\frac{n}{p_2(\cdot)}=\gamma, \frac{n}{\alpha}-\frac{n}{\beta}=\gamma$, $f\in (L^{p_1(\cdot)},L^q)^\alpha (\mathbb{R}^n)$ and $b\in BMO(\R^n)$.
Fix $x\in\mathbb{R}^n$, $r>0$ and set $B=B(x, r)$, $2B= B(x,2r)$. Let $f_1=f\chi_{2B}$ and $f_2=f-f_1$. We write
 \begin{align*}
 |B|^{\frac{1}{\beta}-\frac{1}{p_2(\cdot)}-\frac{1}{q}}\|[b,I_\gamma] f\chi_{B} \|_{L^{p(\cdot)}(\mathbb{R}^n)}
 &\leq|B|^{\frac{1}{\beta}-\frac{1}{p_2(\cdot)}-\frac{1}{q}}\|[b,I_\gamma] f_1\chi_{B} \|_{L^{p_2(\cdot)}(\mathbb{R}^n)}\\
 & ~~~~+|B|^{\frac{1}{\beta}-\frac{1}{p_2(\cdot)}-\frac{1}{q}}\|[b,I_\gamma] f_2\chi_{B} \|_{L^{p_2(\cdot)}(\mathbb{R}^n)}\\
 &=:J_1+J_2.
\end{align*}\par

We first estimate $J_1$. Since the commutators $[b,I_\gamma]$ are bounded from $L^{p_1(\cdot)}(\mathbb{R}^n)$ to $L^{p_2(\cdot)}(\mathbb{R}^n)$ and $\frac{n}{p_1(\cdot)}-\frac{n}{p_2(\cdot)}=\gamma, \frac{n}{\alpha}-\frac{n}{\beta}=\gamma$, we have $\frac{1}{\beta}-\frac{1}{p_2(\cdot)}-\frac{1}{q}=\frac{1}{\alpha}-\frac{1}{p_1(\cdot)}-\frac{1}{q}$ and $-1<\frac{1}{\alpha}-\frac{1}{{p_1}_-}-\frac{1}{q}<\frac{1}{\alpha}-\frac{1}{p_1(\cdot)}-\frac{1}{q}<0$, it implies that
\begin{align*}
J_1&\leq|B|^{\frac{1}{\beta}-\frac{1}{p_2(\cdot)}-\frac{1}{q}}\|[b,I_\gamma] f_1\|_{L^{p_2(\cdot)}(\mathbb{R}^n)}\\
   &\lesssim|b\|_*|B|^{\frac{1}{\beta}-\frac{1}{p_2(\cdot)}-\frac{1}{q}}\| f\chi_{2B} \|_{L^{p_1(\cdot)}(\mathbb{R}^n)}\\
   &\lesssim\|b\|_*|2B|^{\frac{1}{\alpha}-\frac{1}{p_1(\cdot)}-\frac{1}{q}}\| f\chi_{2B} \|_{L^{p_1(\cdot)}(\mathbb{R}^n)}
    \times\bigg(\frac{|B|}{|2B|}\bigg)^{\frac{1}{\alpha}-\frac{1}{{p_1}_-}-\frac{1}{q}}\\
   &\lesssim\|b\|_*|2B|^{\frac{1}{\alpha}-\frac{1}{p_1(\cdot)}-\frac{1}{q}}\|f\chi_{2B} \|_{L^{p(\cdot)}(\mathbb{R}^n)}.
\end{align*}\par
Next we estimate $ I_2$.
\begin{align*}
|[b,I_\gamma]f_2(x)|&=\big|(b(y)-b_B)I_\gamma f_2(y)+I_\gamma \big((b-b_B)f_2(y)\big)\big|\\
                    &\leq\big|(b(y)-b_B)I_\gamma f_2(y)\big|+\big|I_\gamma \big((b-b_B)f_2(y)\big)\big|.
\end{align*}\par
It is obvious that$ |y - z| \approx |x - z|$, when $y \in B(x,r)$ and $z \in B(x,2r)^c$. Decompose $\mathbb{R}^n $ into a geometrically increasing sequence of concentric balls, using the generalized H{\"{o}}lder inequality, we deduce that
\begin{align*}
|(b(y)-b_B)I_\gamma f_2(y)|
                           &\lesssim|(b(y)-b_B)|\int_{2B^c}\frac{|f(z)|}{|x-z|^{n-\gamma}}dz\\
                           &\lesssim|(b(y)-b_B)|\sum_{j=1}^{\infty}|2^{j+1}B|^{\frac{\gamma}{n}-1}\int_{2^{j+1}B\backslash 2^{j}B}|f(z)|dz\\
                           &\lesssim|(b(y)-b_B)|\sum_{j=1}^{\infty}|2^{j+1}B|^{\frac{\gamma}{n}-1}\|f\chi_{2^{j+1}B}\|_{L^{p_1(\cdot)}(\mathbb{R}^n)}
                           \|\chi_{2^{j+1}B}\|_{L^{p_1'(\cdot)}(\mathbb{R}^n)}.\\
|I_\gamma \left((b(y)-b_B)f_2\right)(y)|
                           &\lesssim\int_{2B^c}\frac{|f(z)||(b(y)-b_B)|}{|x-z|^{n-\gamma}}dz\\
                           &\lesssim\sum_{j=1}^{\infty}|2^{j+1}B|^{\frac{\gamma}{n}-1}\int_{2^{j+1}B\backslash 2^{j}B}|f(z)||(b(y)-b_B)|dz\\
                           &\lesssim\sum_{j=1}^{\infty}|2^{j+1}B|^{\frac{\gamma}{n}-1}\|f\chi_{2^{j+1}B}\|_{L^{p_1(\cdot)}(\mathbb{R}^n)}
                           \|(b(y)-b_B)\chi_{2^{j+1}B}\|_{L^{p_1'(\cdot)}(\mathbb{R}^n)}.
\end{align*}\par

From these estimates and using Lemma \ref{l4.1}, Lemma \ref{l4.2}, Lemma \ref{l1.3}, and $-1<\frac{1}{\beta}-\frac{1}{p_2(\cdot)}-\frac{1}{q}=\frac{1}{\alpha}-\frac{1}{p_1(\cdot)}-\frac{1}{q}<0$, $\frac{1}{\beta}-\frac{1}{q}>0$, it follows that\\
\begin{align*}
J_2&\lesssim|B|^{\frac{1}{\beta}-\frac{1}{p_2(\cdot)}-\frac{1}{q}}\sum_{j=1}^{\infty}|2^{j+1}B|^{\frac{\gamma}{n}-1}
        \|f\chi_{2^{j+1}B}\|_{L^{p_1(\cdot)}(\mathbb{R}^n)}\\
   &~~~~\times\bigg(\|(b-b_B)\chi_{B}\|_{L^{p_2(\cdot)}(\mathbb{R}^n)}\|\chi_{2^{j+1}B}\|_{L^{p_1'(\cdot)}(\mathbb{R}^n)}
        +\|(b-b_B)\chi_{2^{j+1}B}\|_{L^{p_1'(\cdot)}(\mathbb{R}^n)}\|\chi_{B}\|_{L^{p_2(\cdot)}(\mathbb{R}^n)}\bigg)\\
   &\lesssim|B|^{\frac{1}{\beta}-\frac{1}{p_2(\cdot)}-\frac{1}{q}}\sum_{j=1}^{\infty}|2^{j+1}B|^{\frac{\gamma}{n}-1}
        \|f\chi_{2^{j+1}B}\|_{L^{p_1(\cdot)}(\mathbb{R}^n)}\\
   &~~~~\times\bigg(\|b\|_*\|\chi_{B}\|_{L^{p_2(\cdot)}(\mathbb{R}^n)}\|\chi_{2^{j+1}B}\|_{L^{p_1'(\cdot)}(\mathbb{R}^n)}
         +\|(b-b_{2^{j+1}B})\chi_{2^{j+1}B}\|_{L^{p_1'(\cdot)}(\mathbb{R}^n)}\|\chi_{B}\|_{L^{p_2(\cdot)}(\mathbb{R}^n)}\\
   &~~~~~~~~+\|(b_B-b_{2^{j+1}B})\chi_{2^{j+1}B}\|_{L^{p_1'(\cdot)}(\mathbb{R}^n)}\|\chi_{B}\|_{L^{p_2(\cdot)}(\mathbb{R}^n)}\bigg)\\
   &\lesssim|B|^{\frac{1}{\beta}-\frac{1}{p_2(\cdot)}-\frac{1}{q}}\sum_{j=1}^{\infty}|2^{j+1}B|^{\frac{\gamma}{n}-1}
        \|f\chi_{2^{j+1}B}\|_{L^{p_1(\cdot)}(\mathbb{R}^n)}\\
   &~~~~\times\bigg(\|b\|_{*}\|\chi_{B}\|_{L^{p_2(\cdot)}(\mathbb{R}^n)}\|\chi_{2^{j+1}B}\|_{L^{p_1'(\cdot)}(\mathbb{R}^n)}
        +(j+1)\|b\|_*\|\chi_{2^{j+1}B}\|_{L^{p_1'(\cdot)}(\mathbb{R}^n)}\|\chi_{B}\|_{L^{p_2(\cdot)}(\mathbb{R}^n)}\bigg)\\
   &\lesssim\|b\|_{*}|B|^{\frac{1}{\beta}-\frac{1}{p_2(\cdot)}-\frac{1}{q}}\sum_{j=1}^{\infty}(j+1)|2^{j+1}B|^{\frac{\gamma}{n}-1}
        \|f\chi_{2^{j+1}B}\|_{L^{p_1(\cdot)}(\mathbb{R}^n)}
       \|\chi_{B}\|_{L^{p_2(\cdot)}(\mathbb{R}^n)}\|\chi_{2^{j+1}B}\|_{L^{p_1'(\cdot)}(\mathbb{R}^n)}\\
   &\lesssim\|b\|_{*}\sum_{j=1}^{\infty}(j+1)|2^{j+1}B|^{\frac{1}{\alpha}-\frac{1}{p_1(\cdot)}-\frac{1}{q}}
        \|f\chi_{2^{j+1}B}\|_{L^{p_1(\cdot)}(\mathbb{R}^n)}
        \frac{\|\chi_{B}\|_{L^{p_2(\cdot)}(\mathbb{R}^n)}}{\|\chi_{2^{j+1}B}\|_{L^{p_2(\cdot)}(\mathbb{R}^n)}}
        \times\frac{|B|^{\frac{1}{\beta}-\frac{1}{p_2(\cdot)}-\frac{1}{q}}}{|2^{j+1}B|^{\frac{1}{\alpha}-\frac{1}{p_1(\cdot)}-\frac{1}{q}}}\\
   &\lesssim\|b\|_{*}\sum_{j=1}^{\infty}(j+1)|2^{j+1}B|^{\frac{1}{\alpha}-\frac{1}{p_1(\cdot)}-\frac{1}{q}}
        \|f\chi_{2^{j+1}B}\|_{L^{p_1(\cdot)}(\mathbb{R}^n)}
        \times\bigg(\frac{|B|}{|2^{j+1}B|}\bigg)^{\frac{1}{\beta}-\frac{1}{q}}
\end{align*}
\begin{align*}
   &\lesssim\|b\|_{*}\sum_{j=1}^{\infty}|2^{j+1}B|^{\frac{1}{\alpha}-\frac{1}{p_1(\cdot)}-\frac{1}{q}}
        \|f\chi_{2^{j+1}B}\|_{L^{p_1(\cdot)}(\mathbb{R}^n)}
        \times\bigg(\frac{(j+1)}{2^{(j+1)n(\frac{1}{\beta}-\frac{1}{q})}}\bigg).
\end{align*}\par
Taking the $L^q$-norm and using Minkowski's inequality, we get
\begin{align*}
&\left\||B|^{\frac{1}{\beta}-\frac{1}{p_2(\cdot)}-\frac{1}{q}}\|[b,I_\gamma]f\chi_{B}\|_{L^{p_2(\cdot)}(\mathbb{R}^n)}\right\|_{L^q(\mathbb{R}^n)}\\
&\leq \|J_1\|_{L^q(\mathbb{R}^n)}+\|J_2\|_{L^q(\mathbb{R}^n)}\\
&\lesssim|b\|_*\bigg(\left\||2B|^{\frac{1}{\alpha}-\frac{1}{p_1(\cdot)}-\frac{1}{q}}\|f\chi_{2B}\|_{L^{p_1(\cdot)}(\mathbb{R}^n)}\right\|_{L^q(\mathbb{R}^n)}\\
&~~~~+\sum_{j=1}^{\infty}\left\||2^{j+1}B|^{\frac{1}{\alpha}-\frac{1}{p_1(\cdot)}-\frac{1}{q}}\|f\chi_{2^{j+1}B}\|_{L^{p_1(\cdot)}(\mathbb{R}^n)}\right\|_{L^q(\mathbb{R}^n)}\left(\frac{(j+1)}{2^{(j+1)n(\frac{1}{\beta}-\frac{1}{q})}}\right)\bigg).
\end{align*}\par
Thus, by taking the supremum over all $ r > 0$, the proof is completed.

Now, assume that $[b,I_\gamma]$ is bounded from $(L^{p_1(\cdot)},L^{q})^{\alpha}(\R^n)$ to $(L^{p_2(\cdot)},L^{q})^{\beta}(\R^n)$. We use the same method as Janson \cite{22}. Choose $0\neq z_0\in\mathbb{R}^n$ such that $0\notin B(z_0,2)$. Then for $x\in B(z_0,2)$, $|x|^{n-\gamma}\in C^{\infty}(B(z_0,2))$. Hence, $|x|^{n-\gamma}$ can be written as the absolutely convergent Fourier series:
$$|x|^{n-\gamma}\chi_{B(z_0,2)}(x)=\sum_{m\in \mathbb{Z}^n}a_me^{2im\cdot x}\chi_{B(z_0,2)}(x),$$
with $\sum_{m\in \mathbb{Z}^n}|a_m|<\infty$.

For any $x_0\in\mathbb{R}^n$ and $t>0$, let $B=B(x_0,t)$ and $B_{z_0}=B(x_0+z_0t,t)$. Let $s(x)=\overline{sgn(\int_{B_{z_0}}(b(x)-b(y))dy)}$. Then
\begin{align*}
\frac{1}{|B|}\int_B|b(x)-b_{B_{z_0}}|dx
=\frac{1}{|B|}\frac{1}{|B_{z_0}|}\int_B\int_{B_{z_0}}s(x)(b(x)-b(y))dydx.
\end{align*}\par
If $x\in B$ and $y\in B_{z_0}$, then $\frac{y-x}{t}\in B(z_0,2)$. Thereby,
\begin{align*}
&\quad\frac{1}{|B|}\int_B|b(x)-b_{B_{z_0}}|dx\\
&=t^{-n-\gamma}\int_B\int_{B_{z_0}}s(x)(b(x)-b(y))|x-y|^{\gamma-n}\left(\frac{|x-y|}{t}\right)^{n-\gamma}\chi_{B}(x)\chi_{B_{z_0}}(y) dydx\\
&=t^{-n-\gamma}\sum_{m\in\mathbb{Z}^n}a_m\int_B\int_{B_{z_0}}s(x)(b(x)-b(y))|x-y|^{\gamma-n}e^{-2im\cdot \frac{y}{t}}\chi_{B_{z_0}}(y)dy\times e^{2im\cdot \frac{x}{t}}\chi_{B}(x)dx\\
&\leq t^{-n-\gamma}\sum_{m\in\mathbb{Z}^n}|a_m|\int_B|[b,I_\gamma](e^{-2im\cdot \frac{\cdot}{t}}\chi_{B_{z_0}})(x)\times s(x)(e^{2im\cdot \frac{\cdot}{t}}\chi_{B})(x)|dx.
\end{align*}\par
By (\ref{7}) and Proposition \ref{p3.1},
$$\frac{1}{|B|}\int_B|b(x)-b_{B_{z_0}}|dx
\lesssim t^{-n-\gamma}\sum_{m\in\mathbb{Z}^n}|a_m|\left\|[b,I_\gamma](e^{-2im\cdot \frac{\cdot}{t}}\chi_{B_{z_0}})\right\|_{(L^{p_2(\cdot)},L^{q})^{\beta}} \left\|s\cdot e^{-2im\cdot \frac{\cdot}{t}}\chi_B\right\|_{\mathcal{H}(p_2(\cdot)',q',\beta')}.$$\par
It is easy to calculate
$$\left\|s\cdot e^{-2im\cdot \frac{\cdot}{t}}\chi_B\right\|_{\mathcal{H}(p_2(\cdot)',q',\beta')}\lesssim t^{\frac{n}{\beta'}}.$$\par
Hence,
$$\frac{1}{|B|}\int_B|b(x)-b_{B_{z_0}}|dx\lesssim t^{-n-\gamma+n/\beta'}\sum_{m\in\mathbb{Z}^n}|a_m|\left\|[b,I_\gamma](e^{-2im\cdot \frac{\cdot}{t}}\chi_{B_{z_0}})\right\|_{(L^{p_2(\cdot)},L^{q})^{\beta}}.$$\par
According to the hypothesis
\begin{align*}
\frac{1}{|B|}\int_B|b(x)-b_{B_{z_0}}|dx&\lesssim t^{-n-\gamma+n/\beta'}\left\|[b,I_\gamma]\right\|_{(L^{p_1(\cdot)},L^{q})^{\alpha}(\R^n)\rightarrow (L^{p_2(\cdot)},L^{q})^{\beta}(\R^n)}\\
&~~~~\times\sum_{m\in\mathbb{Z}^n}|a_m|\left\|e^{-2im\cdot \frac{\cdot}{t}}\chi_{B_{z_0}}\right\|_{(L^{p_1(\cdot)},L^{q})^\alpha}\\
&\le t^{-n-\gamma+n/\beta'+n/\alpha+C_p}\left\|[b,I_\gamma]\right\|_{(L^{p_1(\cdot)},L^{q})^{\alpha}(\R^n)\rightarrow (L^{p_2(\cdot)},L^{q})^{\beta}(\R^n)}\sum_{m\in\mathbb{Z}^n}|a_m|\\
&\lesssim \left\|[b,I_\gamma]\right\|_{(L^{p_1(\cdot)},L^{q})^{\alpha}(\R^n)\rightarrow (L^{p_2(\cdot)},L^{q})^{\beta}(\R^n)}.
\end{align*}\par

Thus, we have
$$\frac{1}{|B|}\int_B|b(x)-b(y)|dx\le\frac{2}{|B|}\int_B|b(x)-b_{B_{z_0}}|dx\lesssim \left\|[b,I_\gamma]\right\|_{(L^{p_1(\cdot)},L^{q})^{\alpha}(\R^n)\rightarrow(L^{p_2(\cdot)},L^{q})^{\beta}(\R^n)}<\infty.$$\par
This prove $b\in BMO(\mathbb{R}^n)$.\par
Thus, the proof is complete.
\end{proof}

\end{document}